\documentclass{article}
\usepackage{lmodern}
\usepackage[leqno]{amsmath}
\usepackage[shortlabels]{enumitem}
\usepackage{amsthm,amssymb,graphicx,yfonts,mathrsfs,textcomp,bbm}
\usepackage{xcolor,soul}

\def\spine{1.1in}
\usepackage[heightrounded, twoside, top=1in, bottom=1.3in, inner=\spine, outer=\spine]{geometry} 
\usepackage[hyperfootnotes=false]{hyperref}
\def\figwidth{0.35\linewidth}

\usepackage{tikz}
\usetikzlibrary {decorations.pathmorphing,arrows.meta}

\def\uline#1{\underline{#1\mkern-6mu}\mkern6mu}

\def\cf{c_{{}_{\rm Fro}} }

\def\epsilon{\varepsilon}
\def\nn{\mathcal N}

\def\m{\mathcal M} 
\def\h{\mathcal H} 

\def\restr#1{\lower4pt\hbox{$\Big|$}\lower9pt\hbox{${\scriptstyle #1}$}\kern-1pt} 
\def\phi{\varphi}
\def\q{{q}}

\def\ge{\geqslant}
\def\geq{\geqslant}
\def\le{\leqslant}
\def\leq{\leqslant}

\def\U{{\mathbf e}}

\def\l{\textswab{L}}
\def\f{\mathcal{F}}
\def\e{\textswab{e}}

\def\weakto{\stackrel{*}{\scalebox{1.8}[1.0]{$ \rightharpoonup $}} } 

\DeclareMathOperator*{\glim}{ \Gamma-lim}

\DeclareMathOperator*{\dist}{dist}

\newcommand{\R}{\mathbb{R}}
 
\newcommand{\bs}[1]{\boldsymbol{#1}} 
\newcommand{\diam}[1]{\text{\rm diam}(#1)}
\newcommand{\diag}{\text{\rm diag\,}}

\newcommand{\omegaNseq}{\{ \omega_N \}_1^\infty}

\newcommand{\supp}{\text{\rm supp\,}}
\newcommand{\Hd}{\mathcal{H}_d}
\renewcommand{\c}{C_{s,d}}
\renewcommand{\d}{\,{d}}

\makeatletter
\renewenvironment{proof}[1][\proofname]{\par
  \pushQED{\qed}%
  \normalfont \topsep6\p@\@plus6\p@\relax
  \trivlist
  \item[\hskip\labelsep
        \itshape
    \textbf{\textit{#1}}\@addpunct{.}]\ignorespaces
}{%
  \popQED\endtrivlist\@endpefalse
}
\providecommand{\proofname}{Proof}

\let\oldfootnote\footnote
\def\footnote{\@ifstar\footnote@star\footnote@nostar}
\def\footnote@star#1{{\let\thefootnote\relax\footnotetext{#1}}}
\def\footnote@nostar{\oldfootnote}
\makeatother

\newtheorem{theorem}{Theorem}
\newtheorem*{theorem*}{Theorem}
\newtheorem{lemma}[theorem]{Lemma}
\newtheorem{proposition}[theorem]{Proposition}
\newtheorem{corollary}[theorem]{Corollary}
\theoremstyle{definition} 
\newtheorem{definition}[theorem]{Definition} 
\newtheorem{remark}[theorem]{Remark}

\numberwithin{theorem}{section}
\numberwithin{equation}{section}

\newlist{gammalist}{enumerate}{1} 
\def\gammalisttag{$  ^{\Gamma} $}
\setlist[gammalist, 1]{label = \textbf{\arabic*\gammalisttag.}, ref = \textbf{\arabic*\gammalisttag}, start=1} 

\title{Asymptotics of  $k$-nearest neighbor Riesz energies}
\author{Douglas P. Hardin \and Edward B. Saff \and Oleksandr Vlasiuk}
\date{}

\begin{document}
\maketitle

\begin{center}
    {\it Dedicated to Ronald DeVore for his 80th birthday}
\end{center}

\begin{abstract}
  We obtain new asymptotic results about systems of $ N $ particles governed by Riesz interactions involving $ k $-nearest neighbors of each particle as $N\to\infty$. These results include a generalization to weighted Riesz potentials with external field. Such interactions offer an appealing alternative to other approaches for reducing the computational complexity of an $ N $-body interaction. We find the first-order term of the large $ N $ asymptotics and characterize the limiting distribution of the minimizers. We also obtain results about the $ \Gamma $-convergence of such interactions, and describe minimizers on the 1-dimensional flat torus in the absence of external field, for all $ N $.
\end{abstract}

\footnote*{{\it Date:}\ \today}
\footnote*{2000 {\it Mathematics Subject Classification.}\ {Primary, 31C20, 28A78; Secondary, 52A40}.}
\footnote*{{\it Key words and phrases.}\ {Riesz energy, $ k $ nearest neighbors,  equilibrium configurations, covering radius, separation distance, meshing algorithms}.}

\section{Introduction and main results}
\label{sec:intro_discrete}

Energy minimization methods for generating unstructured stencils on compact sets in $\R^p$ have been explored in, for example, \cite{borodachovDiscrete2019, hardinMinimal2005, hardinDiscretizing2004}. For a given $d$-dimensional compact set $A\subset \R^p$, these techniques utilize the Riesz kernel $\|x-y\|^{-s}$ with $s>0$  and minimize the following energy associated  with an $N$ point  configuration (tuple) $\omega_N=(x_1,x_2,\ldots, x_N)\in A^N$:
\begin{equation*}
E_s(\omega_N):= \sum_{i=1}^N\sum_{\substack{j=1\\ j\neq i}}^N\|x_j-x_i\|^{-s}.
\end{equation*}  
In the hypersingular case $s\ge d$, the poppy-seed bagel theorem  \cite[Thm 8.5.2]{borodachovDiscrete2019} asserts under mild conditions on $A$ that minimizing configurations $\omega_N^*$ for this energy converge in the weak-star sense to the uniform distribution with respect to $d$-dimensional Hausdorff measure.   More generally, by incorporating a multiplicative weight \cite{borodachovAsymptotics2008} or an external field \cite{hardinGenerating2017} in the above energy, one can generate configurations that converge to a prescribed density on $A$.

An obvious drawback to this method for discretizing manifolds is the  $O(N^2)$ computational  cost     for evaluating the energy or its gradient.    One  approach \cite{borodachovLow2014} to reducing this cost involves radial truncation. Instead, here we analyze truncation of $ E_s $ to a fixed number $k$ of nearest neighbors, as used heuristically in \cite{vlasiukFast2018}.   An advantage of this technique, in contrast to the radial truncation, is that memory and computational costs depend only on $k$ and $N$ (essentially $kN$)   and not on  $\omega_N$.
Furthermore, it leads to dimension-independent methods: optimization of an unweighted Riesz $ k $-energy $ E_s^k $, defined below, for a fixed $ s > 0 $ and $k \geq 1 $ yields a uniform distribution on the underlying set $ A $, irrespective of the Hausdorff dimension of $ A $.

In addition to grid generation,  repulsive interactions depending only on a certain number of nearest neighbors arise in many applications in physics, chemistry, and engineering \cite{isobeHardsphere2015, laiLattice1974, haoStability2013, percusOnedimensional1982}. Inspired by these examples, in the sequel we introduce the Riesz interaction with a fixed number $ k $ of nearest neighbors, and obtain the asymptotics of the minima of the energy $ E_s^k $, as well as the limiting distribution of asymptotic minimizers.

An outline of our manuscript is the following.
In subsections \ref{sec:preliminaries}--\ref{sec:not_conv} we formulate the main results and explain notational conventions assumed for the rest of the paper. Section~\ref{sec:num_experiments} gives some numerical illustrations of applying $ E_s^k $ to discretizing distributions, in particular on manifolds. In Section~\ref{sec:nearest_neighbor} we outline the proof strategy and discuss how choosing nearest neighbors in an interaction influences the geometry of its minimizers. The main proofs are contained in Section~\ref{sec:proofs}. Finally, Section~\ref{sec:connections} begins by investigating the special case of $ A = \mathbb T $, the 1-dimensional flat torus, and finds the minimizers of unweighted $ E_s^k $ on $ \mathbb T $; it also shows that the hypersingular full Riesz interaction $ E_s $ is in a sense limiting case of $ E_s^k $ when $ k\to \infty $. This allows to establish some new results for the hypersingular interaction, namely, the asymptotics of the combined functional, equipped with both weight and external field. The discussion concludes with the proof of $ \Gamma $-convergence of energies $ E_s^k $ on $ N $-point configurations for $ N\to \infty $.

As the authors recently became aware, in the special case of a Jordan measurable\footnote{ A set $A\subset \R^d$ is {\em Jordan measurable} if the   $d$-dimensional Lebesgue measure of the interior of   $A$ and the closure of $A$ are equal.}  set  $A \subset \mathbb R^d $ of positive $d$-dimensional Lebesgue measure,   weight $w\equiv 1$, and no external field, the asymptotic limit of minimizing short-range energies similar to those discussed in this paper was obtained by Fisher   \cite{Fis1964} and for the case of an external field by Garrod and Simmons \cite{GarrodRigorous1972} (see the recent review \cite{Lew2022} for more details).

\medskip
\noindent{\bf Acknowledgments.} The authors express their gratitude to the anonymous referees, whose comments and suggestions helped to improve this article. O.~V.~was supported by an AMS-Simons travel grant.

\subsection{Preliminaries}
\label{sec:preliminaries}
Throughout this paper,  $A$ shall denote a compact set in $\mathbb R^p$ with   $d$-dimensional Hausdorff measure ${\mathcal H}_d(A)<\infty$.  We refer to  an $N$-tuple $\omega_N=(x_1,x_2,\ldots, x_N)\in A^N$ as an {\em $N$-point configuration} (note that this allows repeated points of $A$) and define  the associated normalized counting measure (or empirical measure) 
\begin{equation}
    \label{eq:counting_measure}
\nu(\omega_N):=\frac{1}{N}\sum_{x\in\omega_N} \delta_x,
\end{equation}
with $\delta_x$ the unit mass Dirac measure at $x$.
Defined in this way, the space of $ N $-point configurations $ \omega_N $ inherits the topology from $ (\mathbb R^p)^N $, the latter induced by a fixed (not necessarily Euclidean) norm $\|\cdot\|$ on $\mathbb R^p$. We will occasionally need to apply set-theoretic operations to $ \omega_N $, such as removing or adding entries. For example, by an abuse of notation we write $ \omega_N\setminus \{ x_2 \} := (x_1, x_3, \ldots, x_N) \in A^{N-1} $. In the case of repeated entries in $ \omega_N $, only the first instance is removed. Similarly, $ x\in \omega_N $ means that the point $ x $ is one of the entries of tuple $ \omega_N $ and in this case we further define $ I(x;\omega_N) $ to be the index of the first occurrence of $ x $ as an entry of $ \omega_N $.

For $k,N\ge 1 $ and $\omega_N\in A^N$, let  $ \nn_k(x;\omega_N) $ stand  for the tuple consisting of the  $ k $ nearest neighbors of $x$ from $ \omega_N\setminus \{ x \} $ (with respect to $\|\cdot\|$) where, in the event of ties\label{pg:ties}, we select the points with the smaller indices in $\omega_N$.   When $N<k+1$, then we set  
$ \nn_k(x;\omega_N) =\omega_N\setminus \{ x \} $.
For instance, when $ \omega_4 = (a, a, a, b) $   we have $ \nn_3(a;\omega_4) = \nn_4(a;\omega_4) = ( a, a, b) $. 

For     an  external field  
$V:A\to \mathbb R$, a multiplicative weight  $w:A\times A\to [0,\infty]$, Riesz parameter $s> 0$,  and $k,N\ge 1$, we define the {\it $ k $-nearest neighbor Riesz $s$-energy ($ k $-energy for short)} of an $N$-point configuration  $\omega_N=(x_1,x_2,\ldots, x_N)\in  A^N$ as follows:
\begin{equation}\label{Edef}
    E^k_s(\omega_N; w, V) := 
    \sum_{x\in \omega_N} \sum_{y\in \nn_k(x;\omega_N)}w(x,y ) \|x-y \|^{-s} +  N^{s/d} \sum_{x\in\omega_N} V( x), \qquad k,N \geq 1,\ s > 0.
\end{equation} 
We use the convention that a sum over the empty set is zero; i.e., for $N=1$ we have $E^k_s((x_1); w, V)=V(x_1)$.  For brevity we also  write $ E^k_s(\omega_N; w)$ for $E^k_s(\omega_N; w, 0)$ so that
$$E^k_s(\omega_N; w, V)=E^k_s(\omega_N; w)+N^{s/d} \sum_{x\in\omega_N} V( x).$$

We define the optimal value of the above energy as
\begin{equation*}
    \mathcal{E}^k_s(A,N; w, V) := \inf_{ \omega_N\in A^N}  E^k_s(\omega_N; w, V). 
\end{equation*}
We say that a sequence $ \omegaNseq $  of $N$-point configurations in $A$ is {\em $(k,s,w,V)$-asymptotically optimal} if  
$$\lim_{N\to\infty}\frac{E^k_s(\omega_N; w, V)}{\mathcal{E}^k_s(A,N; w, V)}=1.$$

As in \cite{borodachovAsymptotics2008}, we  require that $w$ be a {\em CPD-weight}; that is,
 $w:A\times A\to [0,\infty]$  satisfies
\begin{enumerate}
    \item[(a)] $ w $ is positive and continuous at $ \h_d $-a.e.\ point of $ \diag(A) $ in the sense of limits taken on $ A\times A $;
    \item[(b)] there is a neighborhood $G \supset \diag(A)$ (relative to $A\times A$) such that $\inf_G w>0$;
    \item[(c)] $w$ is bounded on any closed subset $B\subset A\times A$ such that $B\cap \diag(A)=\emptyset$.
\end{enumerate}
Here CPD stands for {\it (almost) continuous and positive on the diagonal}. 
In fact, for our purposes a weaker version of (a) suffices, assuming (c) be strengthened to the boundedness of $ w $ on the entire $ A\times A $; it will be discussed in Section~\ref{sec:wt_fld}. We shall refer to a weight that satisfies only (b)-(c) as a {\it PD-weight}, for {\it positive on the diagonal}.

We shall refer to a weight $w$ as a {\em marginally radial weight on $A$} if it is of the form $w(x,y)=W(x,\|y-x\|)$ for some  $W:A\times [0,\diam A]\to [0,\infty]$ and $W(x,\cdot)\|\cdot\|^{-s} $ is nonincreasing on $[0,\diam A]$ for each $x\in A$. 
Note that if $w$ is a marginally radial  weight on $A$, then the energy $\mathcal{E}^k_s(A,N; w, V)$ is independent of  the chosen tie-breaking criterion, so this energy is well-defined also when $\omega_N$ is considered as a multiset.   In Theorem~\ref{thm:separation} establishing that a sequence of near energy minimizers has optimal order of separation, we find it convenient to assume that $w$ is a marginally radial weight, since then point energy (potential) is monotonically decreasing as a function of nearest neighbor distances.  

\subsection{Asymptotics of \texorpdfstring{$ k $}{k}-energies}
For the statement of our main results, we use the following definitions and notation:                                     
a set $ A \subset \mathbb R^{p} $ is called {\em $ d $-rectifiable} if for some compact $ A_0       \subset \mathbb R^d $ and a Lipschitz map $ f $ there holds $ A = f(A_0)  $ and $A$ is called {\it $(\mathcal H_d,d)$-rectifiable} if it is a   union of countably many $d$-rectifiable sets together with a set of $\mathcal H_d$-measure zero (see \cite{federerGeometric1996a}). The $ d $-dimensional Lebesgue measure on $\mathbb R^d$ is denoted by $ \mathcal L_d $ and
the $ d $-dimensional Hausdorff measure on $\R^p$ for $d\le p$ is denoted by $ \mathcal H_d $ and is normalized so   as to coincide with $ \mathcal L_d $ on isometric embeddings from $ \mathbb R^d $ to $  \mathbb R^{p} $. 
By $ \|\cdot \| $ we usually denote the Euclidean norm in $ \mathbb R^d $ and $ \mathbb R^p $, but the  arguments below apply to any fixed norms in these spaces. We recall that a sequence of measures $ \mu_n$, $ n\geq 1 $, supported on a compact set $A$ converges weak-star to a measure $ \mu $  on $A$ if
$\lim_{n\to\infty}\int f\, d\mu_n= \int f \, d\mu$ for all continuous $f:A\to \mathbb R$, in which case we write $ \mu_n \weakto \mu $, $ n\to \infty $.

\begin{theorem}
    \label{thm:k_asympt}
    Suppose $ A \subset \mathbb R^{p} $ is a compact $(\Hd,d)$-rectifiable set with $  \Hd(A) = \mathcal{M}_d(A) $,  $s>0$, and $k$ is a positive integer. 
    Then there is a   constant $\c^k$ such that $\c^k>0$ and for any lower semicontinuous external field $V$ and   CPD-weight $w$ the following limit holds:    
    \begin{equation}
        \label{eq:main_asymp}
        \begin{aligned}
            \lim_{N\to\infty}  \frac{ \mathcal{E}^k_s(A,N; w, V)}{N^{1+s/d}}  
            &= \c^k\int_A w(x,x) \rho( x)^{1+s/d}\d\Hd(x) + \int_A V( x) \rho( x) \d\Hd(x)\\
    \end{aligned}
    \end{equation}
    where
    \begin{equation}
        \label{eq:density}
        \rho( x) = \left(\frac{L_1-V( x)}{\c^k(1+s/d)w( x,  x)}\right)^{d/s}_+,\qquad (\cdot)_+:=\max\{0,\cdot\},
    \end{equation} 
    with the constant $ L_1 $  chosen so that $ \rho \d\Hd $ is a probability measure on $A$.
    
    Furthermore, if $ w(x,x) + V(x) $ is finite on a subset of $ A $ of positive $ \h_d $-measure and  $ \omegaNseq $ is a $(k,s,w,V)$-asymptotically optimal sequence of $N$-point configurations in $A$, then  the corresponding normalized counting measures 
    $\nu(\omega_N)$ converge weak-star to $ \rho \d\Hd $.
\end{theorem}

 If $d=p$, note that Theorem~\ref{thm:k_asympt} holds for any compact set $A\subset \mathbb R^p$.  We remark that in the special case $A = \q_d$,  the unit cube in $\mathbb R^d$, $V\equiv0$, and $w\equiv1$, Theorem 1.1 gives
 \begin{equation}\label{CsdkDef0}
 C_{s,d}^k=\lim_{N\to\infty}  \frac{ \mathcal{E}^k_s(\q_d,N; 1, 0)}{N^{1+s/d}}.  
 \end{equation}

\begin{corollary}
 Suppose $ A \subset \mathbb R^{p} $ is a compact $ d $-rectifiable set,  $s>0$, and $k$ is a positive integer.  Let  $\rho:A\to[0,\infty)$ be upper semi-continuous and such that $ \rho \d\Hd $ is a probability measure on $A$.  If $L_1\in \mathbb R$,  $w$ is a CPD-weight on $A\times A$  and $V$ is a lower semi-continuous  external field on $A$ such that 
 \begin{equation*}
 \begin{split}
\frac{L_1- V(x)}{w(x,x)}&=\c^k(1+s/d)\rho(x)^{s/d}, \qquad \text{for $\rho(x)>0$,}\\
 V(x)&\ge L_1, \qquad \text{for $\rho(x)=0$,}
 \end{split}
 \end{equation*}
 then $\nu(\omega_N)\weakto \rho \d\Hd$ for any  $(k,s,w,V)$-asymptotically optimal sequence $ \omegaNseq $. 
 
 In particular, if $V\equiv0$ and 
 \begin{equation*}
w(x,y):= (\rho(x)+  \|x-y\|)^{-s/d} ,
\end{equation*}
then $\nu(\omega_N)\weakto \rho \d\Hd$ for any  $(k,s,w,0)$-asymptotically optimal sequence $ \omegaNseq.$
\end{corollary}

 When  $0<s<d$, it is known (e.g., see \cite{borodachovDiscrete2019}) that  minimizing configurations for non-truncated ($k=N$) Riesz $s$-energy with $w\equiv 1$ and $V\equiv 0$ on a $d$-dimensional set $A$  converge weak-star to the $s$-equilibrium measure on $A$ which is, in general,  not uniform.  
A rather surprising consequence of the above theorem is that the limiting distribution is always uniform (with respect to $\Hd$) for $k$-truncated Riesz $s$-energy, even for $k=1$; i.e. using only   the  one nearest neighbor interaction! 
This provides a basis for applying gradient descent to nearest neighbor truncation of the Riesz energy as a means to obtain a prescribed distribution, a strategy previously used as a heuristic \cite{vlasiukFast2018}. 

Let the quantity
\begin{equation}
    \label{eq:separation}
    \Delta(\omega_N) := \min_{1\leq i < j \leq N} \|x_i - x_j\|
\end{equation}
denote the minimal distance between entries of the configuration $ \omega_N \in  A^N $. We refer to $ \Delta(\omega_N) $ as the {\it separation} of $ \omega_N $.

The following theorem will be necessary to compare the asymptotics of $ k $-energies to those of the full hypersingular Riesz energies; it shows that for $ k $-nearest neighbor interaction with $ k\geq 1 $, near-minimizers are spread over the set $  A $ with the best possible order of separation. In its statement, we say that $ w $ is {\it bounded on $ D $, a subset of $ A $} (as opposed to being bounded on a subset of $ A\times A $), if the values of $ w(z,x) $ and $ w(x,z) $ are bounded uniformly over $ z $ from $ D $ and $ x $  from $ A $:
\begin{equation}
    \label{eq:bounded_w}
    M_w := \sup \{ w(x,z)  : (x,z) \in (A\times D)\cup (D\times A) \} < \infty.
\end{equation}
\begin{theorem}
    \label{thm:separation}
    Suppose $ s > 0 $, $ A \subset  \mathbb R^p $ is compact, $ \mathcal H_d(A) > 0 $, $ w(x,y):A\times A \to [0,\infty) $ is a marginally radial PD-weight, and $ V $ a lower semicontinuous external field, both bounded on some $ D\subset A $, $ \h_d(D) > 0 $. 
    If $ \omegaNseq $ is a sequence such that for some \(R\geq0\),
    \[
        E_s^k(\omega_N; w, V) \leq \mathcal E_s^k(A, N; w, V) + RN^{s/d}, \qquad N\geq 1,
    \]
    then this sequence has the optimal order separation:
    \[
        \Delta(\omega_N) \geq C N^{-1/d}, \qquad N\geq 1,
    \]
    with $ C = C(s,k,p,d,w,V,A,R) $. In addition, in the case $ d=p $, the constant $ C $ can be made independent of the set $ A $.
\end{theorem}

\subsection{Relation to non-truncated hypersingular Riesz energies} 
\label{subsec:relationHyper}
In this section we will clarify the relation between the Riesz $ k $-energy discussed above and the full hypersingular Riesz energy $ E_s(\omega_N; w, V) $ on $ \mathbb R^d $, defined as
\begin{equation}
    \label{eq:hypersingular}
    E_s(\omega_N; w, V) := \sum_{x\neq y\in\omega_N} w(x,y) \|x-y\|^{-s} +  N^{s/d} \sum_{x\in \omega_N} V( x), \qquad  N\geq 2,\ s > d.
\end{equation}
Just as for $ E_s^k $, we write 
\[
    \mathcal E_s(A,N; w, V) := \inf_{\omega_N \in A^N} E_s(\omega_N; w, V).
\]
We will show that for $ k\to \infty $, the asymptotics of $ \mathcal E_s^k $ approach those for $ \mathcal E_s $. 

The asymptotics of the energy \eqref{eq:hypersingular} as given in \eqref{eq:limits_equal} and behavior of its minimizers are known separately for the case of a constant weight and general external field~\cite{hardinGenerating2017},
and for the case of a nonconstant weight in the absence of an external field \cite{borodachovAsymptotics2008}.   
By the general approach outlined in Section~\ref{sec:outline}, these two results can be combined to obtain the analog of Theorem~\ref{thm:k_asympt} for the non-truncated hypersingular interaction with $ s>d $. We derive this result indirectly, by relating the full interaction~\eqref{eq:hypersingular} to the energies $ E_s^k $.  

In the following theorem, we write
 \begin{equation*}
     C_{s,d}:=   \lim_{N\to\infty}  \,\frac{\mathcal E_s(\q_d,N)}{N^{1+s/d}}, \qquad s > d,
\end{equation*}
where $ \q_d $ is the $ d $-dimensional unit cube. By definition, $ C_{s,d} \geq C_{s,d}^k $, $ k\geq 1 $. We call a sequence of configurations $ \omega_N $, $ N\geq 2 $, {\em asymptotically optimal for $ E_s $}, if 
$$\lim_{N\to\infty}\frac{E_s(\omega_N; w, V)}{\mathcal{E}_s(A,N; w, V)}=1.$$

\begin{theorem}
    \label{thm:s_ge_d_asympt}
    If $ A \subset \mathbb R^{p} $ is an $ (\h_d, d) $-rectifiable compact set with $  \Hd(A) = \mathcal{M}_d(A) $, $s>d$, $ w$ is a CPD-weight, and  $ V $ a lower semicontinuous external field, then \begin{equation}\lim_{k\to \infty}\c^k=\c, \end{equation}  and
    \begin{equation}
        \label{eq:limits_equal}
          \begin{aligned}
        \lim_{k\to \infty}\lim_{N\to\infty} 
        & \frac{ \mathcal{E}^k_s(A,N; w, V)}{N^{1+s/d}}  = \lim_{N\to\infty}  \frac{ \mathcal{E}_s(A,N; w, V)}{N^{1+s/d}}\\
        &= \c\int_A w(x,x) \rho( x)^{1+s/d}\d\Hd(x) + \int_A V( x) \rho( x) \d\Hd(x),
        \end{aligned}
    \end{equation}
    where
    \begin{equation*}
        \rho( x) = \left(\frac{L_1-V( x)}{\c(1+s/d)w( x,  x)}\right)^{d/s}_+,
    \end{equation*} 
     with the constant $ L_1 $  chosen so that $ \rho \d\Hd $ is a probability measure on $A$.
     
    Furthermore, if $ w(x,x) + V(x) $ is finite on a subset of $ A $ of positive $ \h_d $-measure and $ \omegaNseq $ is an asymptotically optimal sequence for $ E_s $ on $A$, then  the corresponding normalized counting measures $\nu(\omega_N)$ converge weak-star to $ \rho \d\Hd $.
\end{theorem}
We wish to emphasize that the last equality in \eqref{eq:limits_equal} for the minimal full Riesz interaction energy is also new since it includes both a weight and an external field.

Recall that for $ d=1,\ s> 1 $ it is known,
\begin{equation*}
 C_{s,1}= 2\zeta(s),\quad s>1,
\end{equation*}
where $ \zeta $ is the Riemann zeta function, see e.g. \cite{FiMaRaSa2004}. The {\it universal optimality} of $ E_8 $ and the Leech lattice means that they minimize all energies with completely monotonic kernels as functions of the distance over discrete sets with fixed density. Such optimality of these lattices was shown by Cohn, Kumar, Miller, Radchenko, and Viazovska \cite{cohnUniversal2019}, following the methods of Viazovska \cite{viazovskaSphere2017a}. The Riesz kernel $ 1/r^s $ is completely monotonic (that is, its derivatives have alternating signs), and as a result, $ \c $, $ d=8,24 $, is related to the respective lattice as
\begin{equation}\label{eq:conj_val}
 \c=|\Lambda_d|^{s/d}\zeta_{\Lambda_d}(s),\qquad s>d, \quad d = 8,24.
\end{equation}

Here $ \Lambda_d $  denotes either $ E_8 $ or the Leech lattice; $ |\Lambda_d| $ stands for the volume of the fundamental cell of $ \Lambda_d $, and $ \zeta_{\Lambda_d} $ is the corresponding Epstein  zeta-function.  
The exact value of $ \c $ is unknown for all the other pairs $ s,d $. In dimensions $ d= 2,\ 4$, the conjectured value is also given by the expression \eqref{eq:conj_val} with $ \Lambda_d $, respectively, the hexagonal and $ D_4 $ lattices \cite[Conj. 2]{brauchartNextorder2012}. It is easy to show \cite[Prop. 1]{brauchartNextorder2012} that the  conjectured values \eqref{eq:conj_val} are upper bounds for their respective $ \c $.

\subsection{\texorpdfstring{$\Gamma$}{Gamma}-convergence}
For the hypersingular kernel, uniqueness of the limiting distribution of global minimizers is due to the displacement convexity, in the sense of McCann \cite{mccannConvexity1997}, of the limiting continuous functional (see equation \eqref{eq:gamma_limit} below), which can be obtained by treating $ E_s^k $ as defined on counting probability measures, and then passing to the $ \Gamma $-limit. In the paper~\cite{paper2} we demonstrate that this property is common to all short-range interactions with scale-invariant minimizers. In the present discussion we will derive the $ \Gamma $-limit of $ k $-nearest neighbor energies, as a typical case of a short-range interaction.

We first recall the notion of $ \Gamma $-convergence:
\begin{definition}[\cite{degiorgiSu1975}]
    Let $ X $ be a metric space. Suppose that functionals $ F,\, F_N:X\to \mathbb R,\, N \geq 1, $ satisfy
    \begin{gammalist}
        \item\label{it:gamma_lower} for every sequence $ \{  x_N\}\subset X$ such that $ x_N \to x,\, N\to \infty $, there holds $ \liminf_{N\to\infty} F_N(x_N) \geq F(x) $;

        \item\label{it:recovery_seq} for every $ x\in X $ there exists a sequence $ \{x_N\} \subset X $ converging to it and such that $ \lim_{N\to\infty} F_N(x_N) = F(x) $.  
    \end{gammalist}
    We shall then say that the sequence $ \{F_N\} $ is $ \Gamma $-\textit{converging} to the functional $ F $ on $ X $ with the metric topology; in symbols, $ \displaystyle \glim_{N\to\infty} F_N = F .$
\end{definition}
In our setting, the underlying metric space $ X = \mathcal P(A) $, the space of probability measures on $ A $ with a metric corresponding to the weak$ ^* $ topology; functionals $ F_N(\mu) $ are given by $ E_s^k(\omega_N; w,V) $ when $ \mu = \nu(\omega_N) $ is a counting measure for some $ \omega_N $, see \eqref{eq:counting_measure}, and equal to $ +\infty $ otherwise; see Theorem~\ref{thm:gamma}.
To give the formal definitions, denote by $ \mathcal P_N(A) $ the class of counting measures of $ N $-point tuples in $ A \subset \mathbb R^p $:
\[
    \mathcal P_N(A) := \left\{ \nu(\omega_N) : \omega_N \in A^N \right\}.
\]
In the following result, $ C^k_{s,d} $ is as in~\eqref{CsdkDef0}.

\begin{theorem}
    \label{thm:gamma}
    Suppose $ A \subset \mathbb R^p $ is $ (\h_d,d) $-rectifiable, $ w $ is a  CPD-weight and  $ V $ is a  lower semicontinuous external field.  For $N\ge 1$, let  ${\mathcal F}_N(\cdot; w, V)$ be the functional on $ \mathcal P(A) $  defined   by
    \[
        {\mathcal F}_N(\mu; w, V) := 
        \begin{cases}
            E_s^k(\omega_N;w,V), &\text{ if }\mu = \nu(\omega_N)\in \mathcal P_N(A);\\
            +\infty,                                   & \text{otherwise,}
        \end{cases}
    \]
    and 
    \begin{equation}
        \label{eq:gamma_limit} 
        \f(\mu; w, V) :=   
        \begin{cases}
            C^k_{s,d} \int_A w(x,x)    \rho(x)^{1+s/d}\,d\h_d(x)+
            \int_A V(x) \rho(x) \,d\h_d(x), & \text{if $\mu\ll \h_d$},\\ 
            +\infty, & \text{otherwise},
        \end{cases}
    \end{equation}
    where $\rho$ is the Radon-Nikodym derivative of $\mu$ with respect to $\h_d$.
    Then
    \[
        \glim_{N\to\infty}\frac{{\mathcal F}_N(\cdot\,; w,V) }{N^{1+s/d}} = \f (\cdot\,;w,V)
    \]
    on $ \mathcal P(A) $ equipped  with the weak-star topology.
\end{theorem}

Comparison of Theorem~\ref{thm:gamma} with the classical results for 2-point interactions with integrable kernel reveals the difference in the asymptotic structures of energies for the long-range and short-range energies: limiting functionals of the former depend quadratically (through a double integral) on the limiting measure; on the other hand, in~\eqref{eq:gamma_limit} we have single integrals.

\subsection{Notational conventions}
\label{sec:not_conv}

Let us summarize the notation introduced in previous sections, as well later in the paper.
It is assumed that $ p, d $ are integer with $ p \geq d > 0 $. By $ \|\cdot \| $ we denote a fixed norm on $ \mathbb R^p $, not necessarily Euclidean, as well as its restriction to $ \mathbb R^d $, which we treat as a subset of $\mathbb R^p$.
Closed balls in the ambient space (either $ \mathbb R^d $ or $ \mathbb R^p $) with respect to these norms are denoted by $ B(x,r) $; here  $ x $ is the center of the ball, $ r $ stands for the radius. For $r>0$, the closed $r$-neighborhood of a compact set $ A $ is denoted by $ A_r = \bigcup_{x\in A} B(x,r) $. Notation $ v_d $ stands for the volume of the unit ball in $ \mathbb R^d $. 

A ``cube'' always refers to a closed cube with sides parallel to the coordinate axes. The unit cube in $ \mathbb R^d $, centered at the origin, is denoted by $ \q_d = [-1/2, 1/2]^d $.                                      

The $ d $-dimensional Lebesgue and Hausdorff measure are denoted by $ \mathcal L_d $ and $ \mathcal H_d $; the latter is normalized so as to coincide with $ \mathcal L_d $ on isometric embeddings from $ \mathbb R^d $ to $  \mathbb R^{p} $. 
Weak$ ^* $ convergence of a sequence of measures $ \mu_n$, $ n\geq 1 $, to $ \mu $ is denoted by $ \mu_n \weakto \mu $, $ n\to \infty $. 
Notation $ \mathcal M_d $ stands for the $ d $-dimensional Minkowski content in $ \mathbb     R^{p} $. 

The adjacency graph of $ \omega_N $, introduced in Section~\ref{sec:outline} and corresponding to the nearest neighbor relation, is denoted by $ \Lambda_k(\omega_N) $. Notation $ \prec_x $ stands for the ordering of points in $ \omega_N $ by indices and distance to a given point $ x\in\mathbb R^p $. The $ l $-th element of $ \omega_N \setminus \{ x \} $ under the ordering $ \prec_x $ is written as $ (x;\omega_N)_l $ (note that the set difference here removes only the first occurrence of $ x $ in $ \omega_N $). Given $ x\in\omega_N $, we write $ I(x;\omega_N) $ for the index of the first occurrence of $ x $ as an entry of tuple $ \omega_N $.

A bijective map $ \psi:\mathbb R^d\to \mathbb R^p $ is said to be bi-Lipschitz with the constant $ (1+c) $, $ c> 0 $, if there holds
\[
    (1+c)^{-1} \|x-y \| \leq \|\psi(x) - \psi(y) \| \leq (1+c) \|x-y \|
\]
for every pair $ x,y\in \mathbb R^d $.

In cases when the multiplicative weight and/or external field are absent from our             considerations, we write simply $ E^k_s(\omega_N; w) $ and $E^k_s(\omega_N)$ in place of $ E^k_s(\omega_N; w, 0) $ and $ E^k_s(\omega_N; 1, 0) $, respectively. Finite positive constants that may depend on some arguments are denoted $ C(\ldots) $; we can sometimes refer to different constants of this form in different parts of an equation, using the same symbol $ C $.

\section{Numerical aspects and experiments}
\label{sec:num_experiments}
\begin{figure}[ht]
    \centering
    \includegraphics[width=\figwidth]{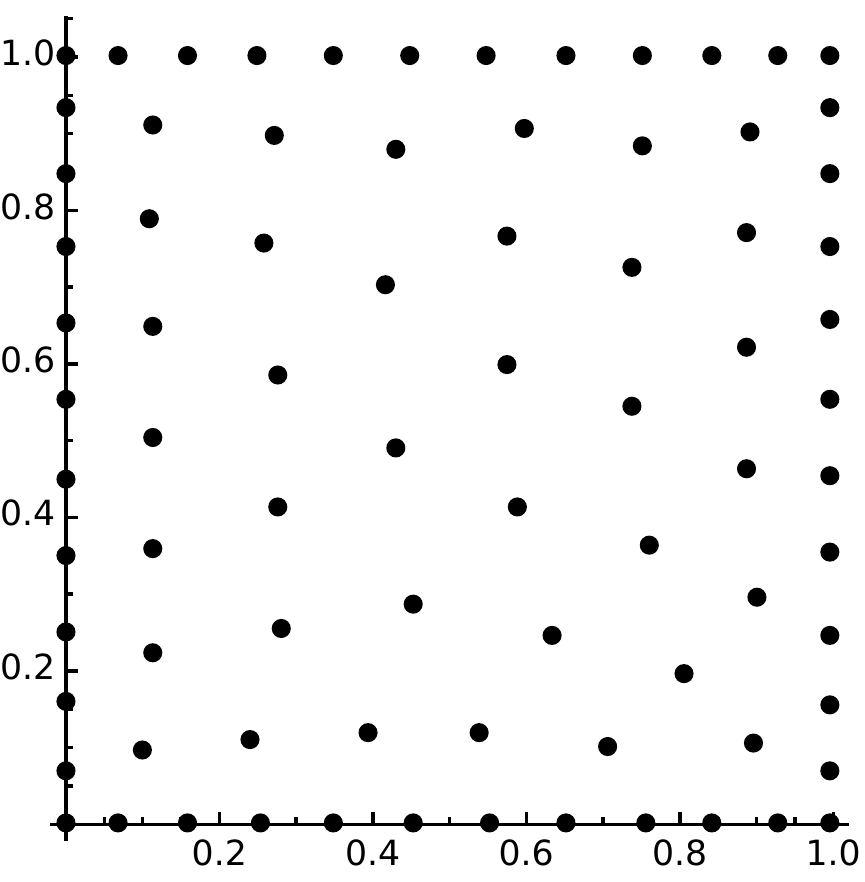}
    \hspace{.05\linewidth}
    \includegraphics[width=\figwidth]{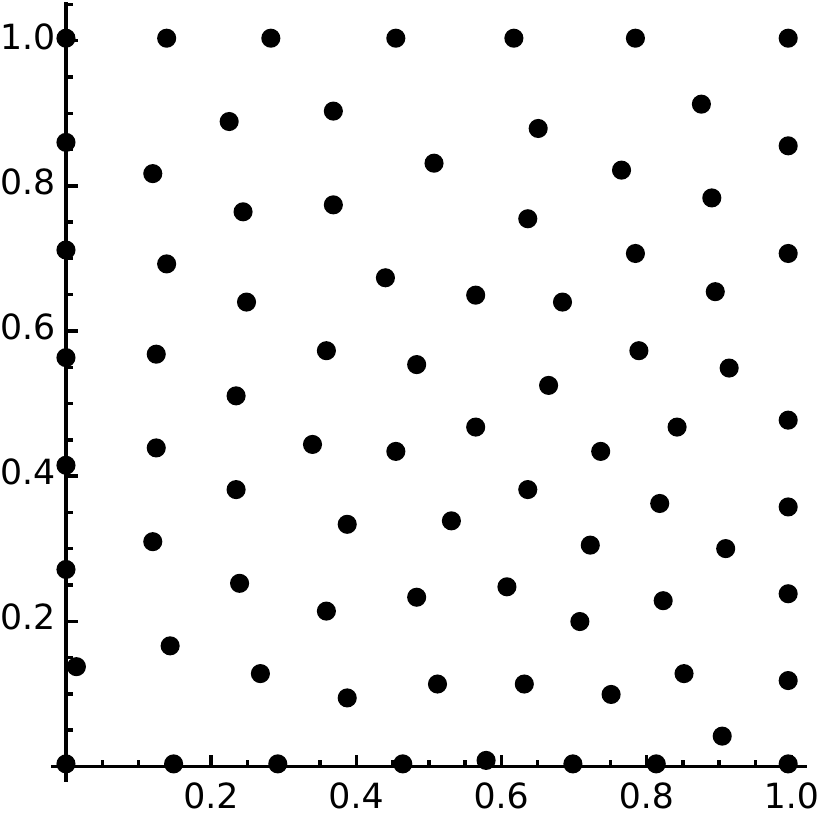}
    \caption{Left: approximate minimizer of the full Riesz interaction with $ s=1 $; right: approximate minimizer of the $ k $-nearest neighbor interaction with $ k=2 $ and $ s=1 $. In both images, $ N=80 $.}%
    \label{fig:squares}
\end{figure}
\begin{figure}[ht]
    \centering
    \includegraphics[width=\figwidth]{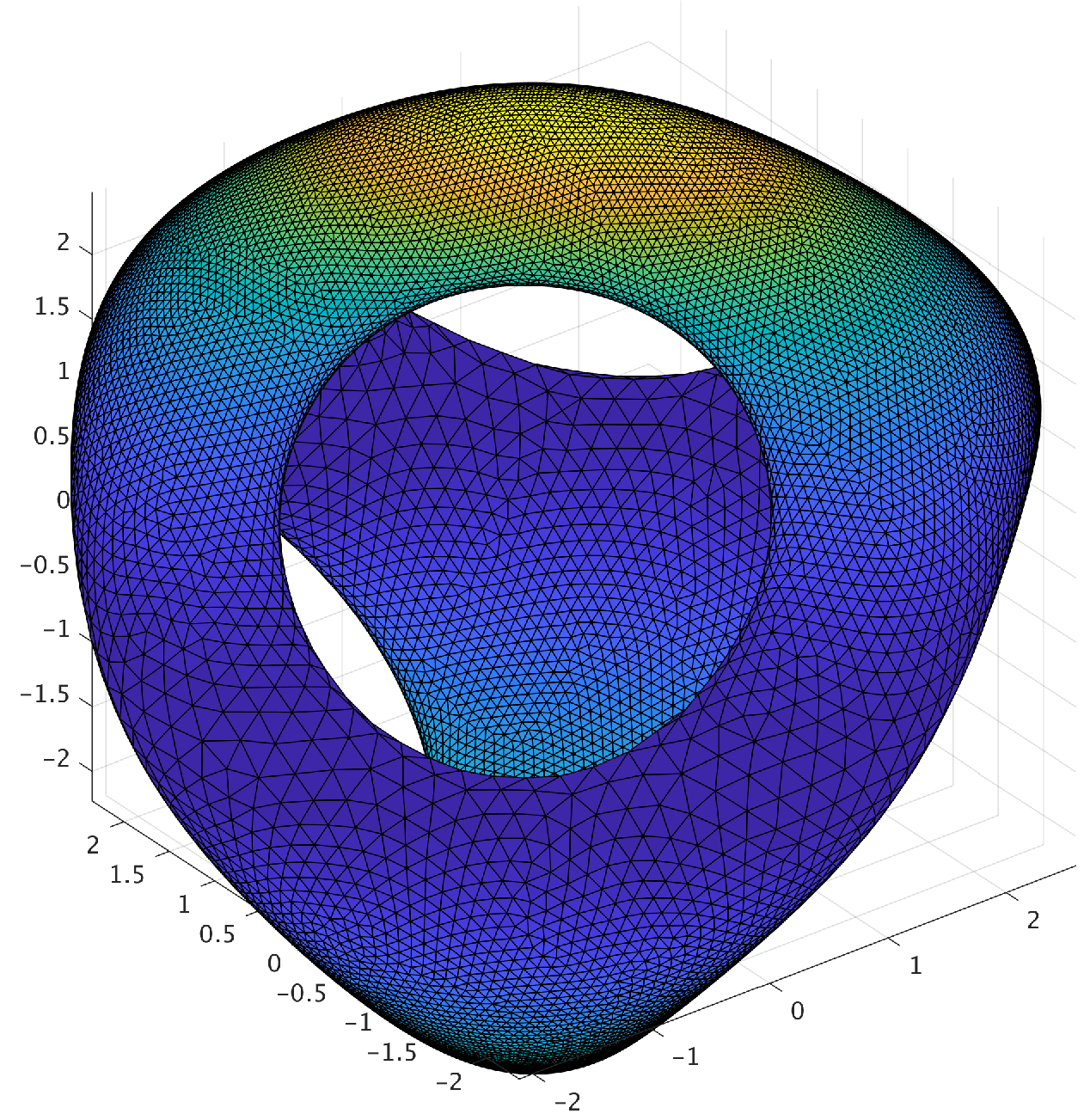}
    \hspace{.05\linewidth}
    \includegraphics[width=\figwidth]{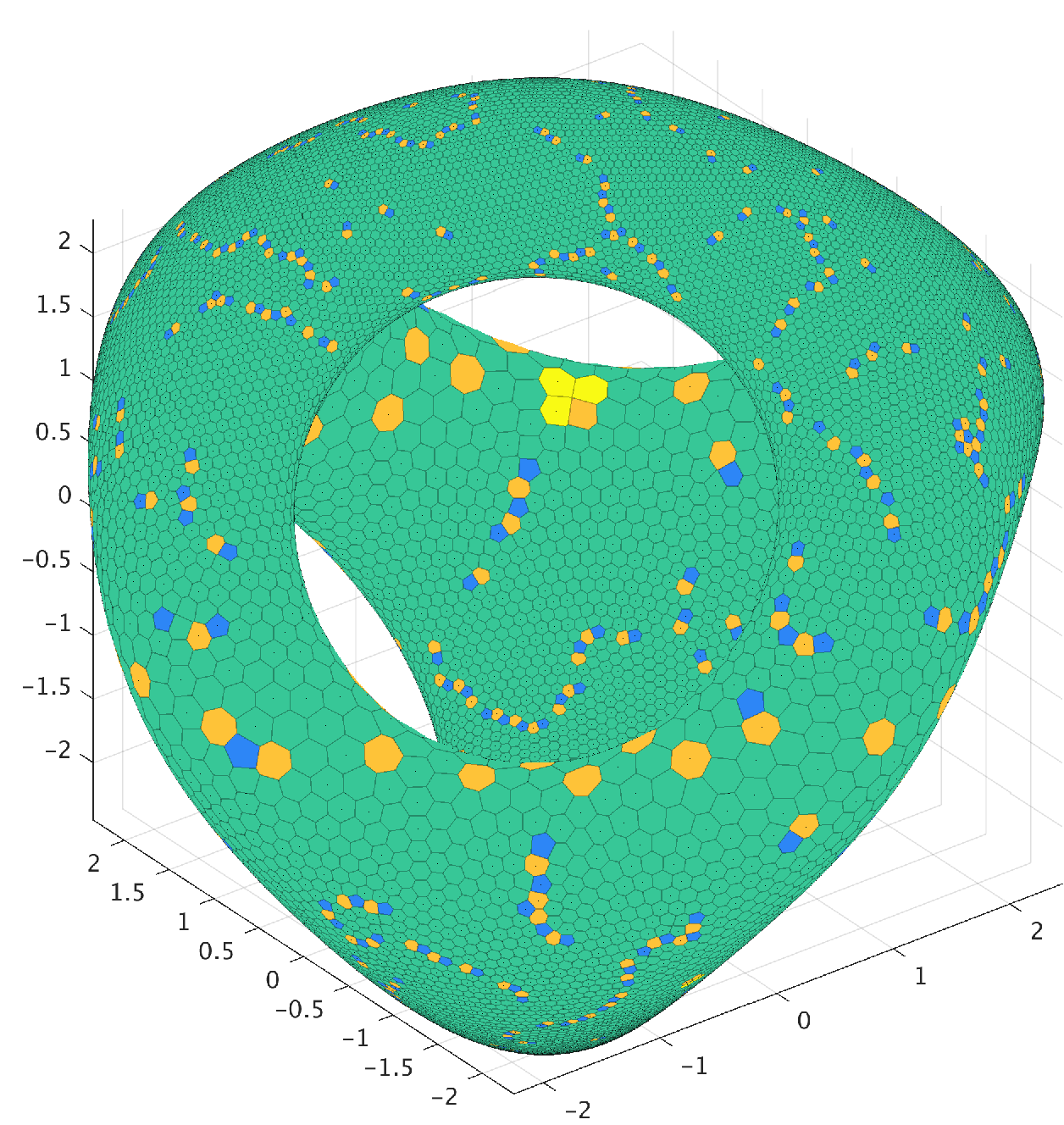}
    \caption{(Color online.) Approximate minimizing configuration of 20,000 points on a genus 3 surface in the $x_1x_2x_3$-cube $[-2,2]^3$ for the energy $E^k_s$ with $s=4$,  $k=30$ and weight $w$ chosen to give density $ \rho $ proportional to $ x_3^2 $. External field $V\equiv0$. Left: Delaunay triangulation; right: Voronoi tessellation. }%
    \label{fig:handles}
\end{figure}

 The worst case complexity 
for constructing the $k$-nearest neighbor ($k$-nn) graph for an $N$-point configuration is $O(N \log N)$ floating point computations (FLOPS). Thus,  the cost of evaluating $E_s^k$ (or its gradient) is also $O(N\log N)$ compared with  $O(N^2)$ FLOPS for the full interaction energy $E_s$.
 However, for sufficiently well-separated point configurations, the  $k$-nn algorithm  reduces to $O(kN)$ and thus the cost of one energy or gradient evaluation is also $O(kN)$.   

An interesting question is how the choice of $ k $ influences the speed of convergence of a given optimization algorithm such as gradient descent (we remark that $E_s^k$ is not differentiable when there are ties for the $k$-th nearest neighbor). 
This  issue is not explored here, but is left for future investigations. 

In Figure~\ref{fig:squares}, we show approximate energy minimizers for $N=80$ points in the unit square for the full $E_s$ energy (left) and  $E_s^k$ energy (right) with $s=1$, $k=2$, $w\equiv 1$, and $V\equiv 0$.    Notice that the full kernel energy yields higher density distribution at the boundary of the square while, in accordance with Theorem~\ref{thm:k_asympt}, the points are more uniformly spaced when the interactions are restricted to the $k$-nearest neighbors. The two approximate minimizers were computed using Mathematica's IPOPT interface and the built-in simulated annealing algorithm, respectively.

It has been demonstrated that the energies $ E_s^k $ can be used for efficient discretization of complicated surfaces, see \cite{vlasiukFast2018, vlasiukOVlasiuk}.  Here we illustrate the effectiveness of the algorithm in Figure~\ref{fig:handles} which shows an approximate minimizing configuration of $ N =20\,000 $ points on an algebraic surface with $s=4$, $k=30$, $V\equiv 0$, and a nonuniform weight.  The left image  shows 
the Delaunay triangulation of the configuration colored according to point  density, where lighter colors reflect higher density. The right image shows the Voronoi tessellation of the surface generated by the configuration. Cells of the tessellation are colored according to their number of edges. Notice that the majority of cells are hexagons (light green).

\section{Geometry of nearest neighbor interactions}
\label{sec:nearest_neighbor}

\subsection{Proof strategy and adjacency graph of nearest neighbors}
\label{sec:outline}
Our strategy, as put forward in \cite{paper2}, is to show that the unweighted functional $ E_s^k $ is a so-called short-range interaction, that is, it has the following four essential properties. Note that compared to paper~\cite{paper2}, we strengthen and simplify the formulations, as appropriate for our context.
\begin{itemize}
    \item[(i)] \ul{Monotonicity}: If $ A\subset B \subset \mathbb R^p$, then, by definition,
    \begin{equation}
        \mathcal E^k_s(A,N)  \geq \mathcal E^k_s(B,N),\qquad  N\ge 1.
    \end{equation}
    \item[(ii)] \ul{Asymptotics on cubes}:    For the unit cube $  \q_d \in \mathbb R^d$, the following limit exists and is positive and finite 
        \begin{equation}
            C_{s,d}^k := \lim_{N\to \infty} \frac{\mathcal E_s^k(\q_d, N)}{N^{1+s/d}}.
        \end{equation}
        This fact will be established in Lemma~\ref{thm:cube}.
    \item[(iii)] \ul{Short-range property}: The energy of a sequence of configurations contained in a pair (or finite collection) of disjoint compact sets   is asymptotically the sum of energies on individual sets. Suppose $ A_1, A_2 \subset \mathbb R^p $ are disjoint compact sets.   If $ \omegaNseq$ is a sequence of $N$-point configurations in  $ A_1\cup A_2 $ for $ N \geq 1 $,   then 
        \begin{equation}
            \label{eq:short_range}
            \lim_{N\to\infty} \frac{  E_s^k(\omega_N\cap A_1)+E_s^k(\omega_N\cap A_2)}{E_s^k(\omega_N)} = 1.
        \end{equation}
        The short-range property will be obtained in Lemma~\ref{lem:short_range}.
    \item[(iv)] \ul{Stability}: The minimum energy asymptotics  is stable under small perturbations (in terms of Minkowski content) of the set; that is,
    for every compact $ A\subset \mathbb R^p $ and $ \epsilon \in (0,1) $ there is some $\delta=\delta(\epsilon, s,k,p,d,A)>0$ such that for any  compact $ D \subset A $ satisfying $ \mathcal M_d(D) \geq (1 - \delta)\,\mathcal M_d(A)$, we have
    \begin{equation}
        \label{eq:set_conts}
        \liminf_{N\to \infty}\frac{\mathcal E^k_s(A,N)}{N^{1+s/d}} \geq (1-\epsilon)
        \liminf_{N\to \infty} \frac {\mathcal E^k_s(D,N)}{N^{1+s/d}}, \quad 
        \limsup_{N\to \infty}\frac{\mathcal E^k_s(A,N)}{N^{1+s/d}} \geq (1-\epsilon)
        \limsup_{N\to \infty} \frac {\mathcal E^k_s(D,N)}{N^{1+s/d}}.
    \end{equation}
    In addition, for $ p=d $, $ \delta $ is independent of $ A $. This result will be established in Lemma~\ref{lem:stable}.
\end{itemize}
Once these properties have been established for the unweighted interaction $ E_s^k $, existence of the asymptotics on compact sets in $ \mathbb R^d $ follows as an extension of the statement (ii) for cubes. 
The statement of Theorem~\ref{thm:k_asympt}, applying to $ ({\mathcal H}_d,d) $-rectifiable sets $ A\subset \mathbb R^p $ with $ \mathcal H_d(A) =\mathcal M_d(A) $, is then derived by approximating such sets with bi-Lipschitz parametrizations, an argument going back to Federer~\cite{federerGeometric1996a}.

\bigskip

For some of the proofs in the sequel it will be useful to order the entries of    $ \omega_N\in (\mathbb R^p)^N $   by their distance to a given $ x\in \mathbb R^p $; as was mentioned in Section~\ref{sec:intro_discrete}, the interaction $ E_s^k $ selects points with smaller indices in the case of equal distance, so we will order lexicographically, first by distance, then by index. Formally, the order $ \prec_x $ on entries of $ \omega_N $ is defined like so:
\[
    y\prec_x z\qquad \stackrel{def}{\iff} \qquad 
    \begin{aligned}
        \|y-x\| &< \|z-x\|\\
        \text{ or }\\
        \|y-x\| &= \|z-x\| \text { and } I(y;\omega_N) < I(z;\omega_N),
    \end{aligned}
\]
where as before, $ I(y;\omega_N) $ is the index of the first occurrence of $ y $ as an entry of $ \omega_N $. 
The notation  $ \nn_k(x;\omega_N) $  introduced above then stands for the tuple of the first $ k $ entries of $ \omega_N\setminus \{ x \} $ with respect to the ordering $ \prec_x $. We further write $ (x;\omega_N)_l $ for the $ l $-th entry of $ \omega_N \setminus \{x\} $ with respect to $ \prec_{x} $, $ 1\leq l \leq N-1 $. In particular, distances $ \|x - (x;\omega_N)_l\|  $ are nondecreasing in $ l $ for a fixed $ x $ and $ \omega_N $.

Let
\[
    \Lambda_k(\omega_N) := \{ (x,y) : x,y\in\omega_N, \ y\in \nn_k(x;\omega_N) \},
\]
the set of ordered pairs of entries of $ \omega_N $, corresponding to the relation ``$y$ is among the $ k $ nearest neighbors of $ x $''. Notice that this relation is not symmetric. In what follows, it will be occasionally convenient to think of $ \Lambda_k(\omega_N) $ as the set of edges in the oriented graph $ (\mathcal V, \mathcal E) = \left(\{ x_i \}_1^N,\ \Lambda_k(\omega_N)\right) $ with $ \{ x_i \}_1^N $ being the multiset of entries from $ \omega_N $. Due to this, we will refer to $ \Lambda_k $ as the {\it adjacency graph} of $ \omega_N $.

\subsection{Main geometric lemma and local properties of near-minimizers}
Using $ \Lambda_k(\omega_N) $, we can write
\begin{equation}\label{eq:k_energy}
    E^k_s(\omega_N; w, V) = \sum_{(x,y)\in\Lambda_k(\omega_N)} w( x,  y) \|x - y \|^{-s} +  N^{s/d} \sum_{x\in \omega_N} V(x), \qquad  s \neq 0.
\end{equation}
As before, the function $ w $ is assumed to be a CPD-weight on $ A \times A $. 
The external field $ V $ is assumed to be lower semicontinuous on $ A $ (and therefore bounded below there). 

Our eventual goal is to verify the properties from Section~\ref{sec:outline}. It is easy to see that restricting interactions to $ k $ nearest neighbors guarantees that $ E^k_s $ is in a sense local. Without such restriction, the locality does not hold when $ s < d $, as is well-known from classical potential theory. Since $ s> 0 $, the singular nature of the interaction on the diagonal results in that the pointwise separation is of the optimal order for near-minimizers, as will be shown in Theorem~\ref{thm:separation}.

\medskip

We will first obtain the following basic fact about the set $ \Lambda_k(\omega_N) $.
\begin{lemma}
    \label{lem:few_nns}
    Fix a configuration $ \omega_N \in (\mathbb R^d)^N $ of $ N $ distinct points. For any $ y\in \omega_N $, the number of points $ x $ in $ \omega_N $ such that $ y $ is one of $ k $ nearest neighbors of $ x $ is bounded by $ n(k,d) $, depending only on the number of neighbors $ k $ and the dimension $ d $. That is, 
    \[
        \#\{ x\in \omega_N : y\in \nn_k(x;\omega_N) \} \leq n(k,d).
    \]
\end{lemma}
This lemma can be interpreted in graph-theoretic terms as follows. Consider expression~\eqref{eq:k_energy}; the first sum involves terms $ w(x,y) \|x-y\|^{-s}  $ for oriented pairs $ (x,y) \in \Lambda_k(\omega_N) $. By definition, the outgoing degree of every vertex in the graph $\left(\{ x_i \}_1^N,\ \Lambda_k(\omega_N)\right)$ is $k$; the above lemma shows further that the maximal incoming degree in the graph is bounded by $ n(k,d) $. It is also useful to note that $ y $ does not have to be an element of $ \omega_N $ for the result to hold.
\begin{proof}
    Fix $ y \in \omega_N $ and denote  $ \omega_{N,y} = \{ x\in \omega_N : y\in \nn_k(x;\omega_N) \} $. Choose the radius $ r_y> 0 $ so that $ B(y,r_y) $ does not contain any points from $ \omega_N $ except $ y $.  Let $ \pi_S $ be the radial projection onto the sphere $ S := \partial B(y,r_y) $ and consider the image $ \pi_S(\omega_{N,y})$, see Figure~\ref{fig:nns}.
    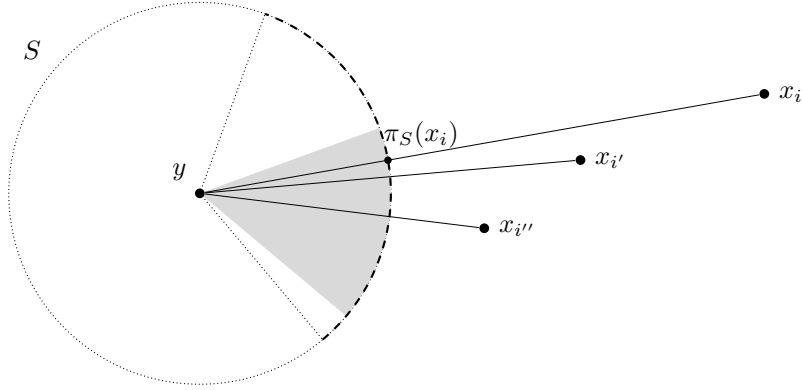
\begin{figure}[t]
        \centering
        \begin{tikzpicture}
            \draw [densely dotted] (0,0) circle [radius=1in];
            \path [fill=gray!30] (0,0) -- (-40:1in) arc (-40:20:1in) -- (0,0) ;

            \draw [dashed, thick] (-50:1in)  arc (-50:70:1in) ;
            \draw [densely dotted] (0,0) -- (70:1in);
            \draw [densely dotted] (0,0) -- (-50:1in);
            
            \node [inner sep=1pt, label=140:$S$] at (140:1in)  {};

            \node [circle, fill, inner sep=1.3pt, label=140:${y}$] (xj) at (0,0)  {};
            \node [circle, fill, inner sep=1.3pt, label=0:$x_i$] (xi) at (10:3in)  {};
            \node [circle, fill, inner sep=1.3pt, label=0:$x_{i'}$] (xip) at (5:2in)  {};
            \node [circle, fill, inner sep=1.3pt, label=0:$x_{i''}$] (xipp) at (-7:1.5in)  {};
            \node [circle, fill, inner sep=1pt] (pixi) at (10:1in)  {};
            \node at (15:1.2in) {$\pi_S(x_i)$};

            \draw [very thin] (xj)  -- (xi);
            \draw [ultra thin] (xj)  -- (xip);
            \draw [ultra thin] (xj)  -- (xipp);
        \end{tikzpicture}
        \caption{The open spherical cap of radius $ \pi/3 $ around projection $ \pi_S (x_i) $ (dashed) contains the projection of $ x_{i'} $. As Lemma~\ref{lem:few_nns} shows, at most $ k $ points among $ \{ x\in \omega_N : y\in \nn_k(x;\omega_N) \} $ can be projected into any given cap of angular radius $ \pi/6 $ (shaded).  }
        \label{fig:nns}
    \end{figure}
    Suppose that an open spherical cap $ B_S(z,\pi/6) $ on $ S $ of angular radius $ \pi/6 $ and center $ z\in S $ contains more than $ k $ elements of this image. Let
    $$ x_i = \arg\max \{ \|x-y\| :x\in \omega_{N,y},\ \pi_S(x)\in  B_S(z,\pi/6) \}. $$ 
    Then $ B_S(z,\pi/6) \subset B_S(\pi_S(x_i),\pi/3) $, implying for any $ x_{i'} \in \omega_{N,y} $, from $ \pi_S(x_{i'}) \in B_S(z,\pi/6) $ it follows $ \angle x_iyx_{i'} < \pi/3 $, so that
    $$ \|x_{i'}-x_i\| < \max\{ \|y-x_i\|, \|y-x_{i'}\|\} = \|y-x_i\|, $$
    since $ x_i $ was chosen the furthest from $ y $. Thus, every other point projected into $ B_S(\pi_S(z),\pi/6) $ is closer to $ x_i $ than $ y $, and it must be $ y \notin \nn_k(x_i; \omega_N) $, a contradiction.
    By this argument, the constant $ n(k,d) $ chosen as
    \[
        n(k,d) : = \max \left\{ n : \exists\, \omega_n \in (\mathbb S^{d-1})^n \text{ such that } \#(\omega_n\cap B_S(z,\pi/6))\leq k,\ \forall z\in \mathbb S^{d-1}\right\}
    \]
    has the properties stated in the claim of the lemma.
\end{proof}
\noindent We will also need a classical result from potential theory, due to Frostman.
\begin{proposition}[Frostman's lemma {\cite[p. 112]{mattila1995geometry}}, \cite{frostman1935potentiel}]
    \label{prop:frostman}
    For any compact $ A \subset \mathbb R^p $ with $ \mathcal H_d(A) > 0 $, there is a finite nontrivial Borel measure $ \mu  $ on $ \mathbb R^{p} $ with support inside $ A $ such that
\[ 
    \mu(B(x,r)) \leq r^d,\qquad {x}\in \mathbb R^{p},\ r>0.
\]
\end{proposition}
This statement is used to obtain a lower bound on the optimal covering radius of the compact set $ A $. Indeed, let the measure $ \mu $ be as in Frostman's lemma. Given any collection $ \omega_N  =  (x_i)_1^N \in A^N $, for the radius $ r_0=  \cf(A) N^{-1/d} := (\mu(A)/2)^{1/d} N^{-1/d}  $ and the set
$$ D =A\setminus \bigcup_{i} B({x}_i, r_0) $$
there holds $ \mu (D) \geq \mu(A)/2 $. In particular $ D \neq \emptyset $, so that at least one point of $ A $ is distance $ r_0 $ away from the points in $ \omega_N $. It follows that the covering radius of $ A $ for any collection of $ N $ points $ \omega_N\in A^N $ is at least $ \cf(A) N^{-1/d} $. Observe also that for $ d=p $, it suffices to use $ \mu = v_d^{-1}\mathcal L_d $, and hence $ \mu $ is independent of $ A $ in this case.

\begin{proof}[Proof of Theorem~\ref{thm:separation}]
    Fix an $ N > 2 $. Since $ s > 0 $ and the product $ w(x,y)\|x-y\|^{-s} $ is infinite on the diagonal of $ A\times A $, assumptions of the theorem imply that all entries of $ \omega_N $ are distinct. It will be convenient to assume that  configuration $ \omega_N = ( x_1,\ldots,x_N ) $ is numbered in such a way that minimal separation is attained for the pair $ x_1,x_2 $:
    \[
        \Delta(\omega_N) = \| x_1 - x_2\| = c_N N^{-1/d}.
    \]
    It will also be convenient to assume $ V \geq 0 $ on $ A $; by lower semicontinuity this can be achieved by adding a sufficiently large constant to $ V $, which does not change the behavior of minimizers.

    We need to show $ c_N \geq C(s,k,d,w,V,A,R) > 0 $, $ N\geq 1 $.  
    By definition, a PD-weight satisfies properties (b)-(c) of a CPD-weight; thus there exists a $ \delta > 0 $ such that $ \{ (x,z) : \|x-z\|\leq \delta \} \subset G $ for the neighborhood $ G $ as in the definition of CPD-weight. This implies
    \[
        0 < m_w : =  \inf \{ w(x,z) : \|x-z\|\leq \delta \}. 
    \]
    Let further $ M_w $ be as in \eqref{eq:bounded_w} and $ L = \sup_D V $ -- both finite quantities, due to the boundedness of $ w $ and $ V $ on $ D $.

    Since $ D $ is a set of positive $ \h_d $-measure, and in view of the discussion after  Proposition~\ref{prop:frostman}, there exists a point $ z \in D $, such that
    \[
        \| z - x_i\| \geq \cf N^{-1/d}, \qquad i = 1,\ldots,N,
    \]
    where $ \cf := c(D) =  (\mu(D)/2)^{1/d} $ is from Frostman's lemma for $ D $. Let 
    $$ \omega_N' =  (z,x_2,\ldots,x_N),   $$
    the configuration obtained by replacing $ x_1 $ with $ z $. If $\| x_1 - x_2\| = c_N N^{-1/d} \geq \delta$,
    there is nothing to prove. Otherwise, suppose $ c_N N ^{-1/\d} < \delta $ for some $ N $. 
    Since $ E^k_s(\omega_N; w) $ is close to being optimal, replacing  $ x_1 $ with $ z $ can lower the value of $ E_s^k $ by at most $ RN^{s/d} $:
    \begin{equation*}
        E^k_s(\omega_N'; w,V) - E^k_s(\omega_N; w,V) = E^k_s(\omega_N'; w) - E^k_s(\omega_N; w) + N^{s/d}(V(z)-V(x_1)) \geq -R N^{s/d},
    \end{equation*}
    so that, since $ V(z) - V(x_1) \leq V(z) \leq L $,
    \begin{equation}
        \label{eq:almost_min}
        E^k_s(\omega_N'; w) - E^k_s(\omega_N; w) \geq -(R+L) N^{s/d}.
    \end{equation}
    Let us determine the terms remaining after cancellation in the left-hand side of this equation. Since 
    \[
        E^k_s(\omega_N'; w) - E^k_s(\omega_N; w) = \left(\sum_{(x,y) \in \Lambda_k(\omega_N')} - \sum_{(x,y) \in \Lambda_k(\omega_N)}\right) w(x,y) \|x-y\|^{-s},
    \]
    we can describe all the terms that occur only in $ E_s^k(\omega_N;w) $, and therefore do not cancel out, as
    \[
        \Sigma_1 + \Sigma_2 + \Sigma_3 := 
        \left(
        \sum\limits_{\substack{ (x,y) \in \Lambda_k(\omega_N),\\ x=x_1 }}
        +
        \sum\limits_{\substack{ (x,y) \in \Lambda_k(\omega_N),\\ y=x_1 }}
        +
        \sum\limits_{\substack{ (x,y) \in \Lambda_k(\omega_N)\setminus\Lambda_k(\omega_N')\\ x,y \neq x_1 } }
        \right)
        w(x,y) \|x-y\|^{-s}.
    \]
    Likewise, the terms occurring only in $ E_s^k(\omega_N';w) $ are as follows:
    \[
        \Sigma_4 + \Sigma_5 + \Sigma_6 := 
        \left(
        \sum\limits_{\substack{ (x,y) \in \Lambda_k(\omega_N'),\\ x=z }}
        +
        \sum\limits_{\substack{ (x,y) \in \Lambda_k(\omega_N'),\\ y=z }}
        +
        \sum\limits_{\substack{ (x,y) \in \Lambda_k(\omega_N')\setminus\Lambda_k(\omega_N)\\ x,y \neq z } }
        \right)
        w(x,y) \|x-y\|^{-s}.
    \]
    The term $ \sum_3 $ above arises due to the number of terms originating from each point being fixed at $ k $, so any terms incoming into $ z $ must have had different terminating nodes in $ \omega_N $; similar logic applies to $ \sum_6 $.
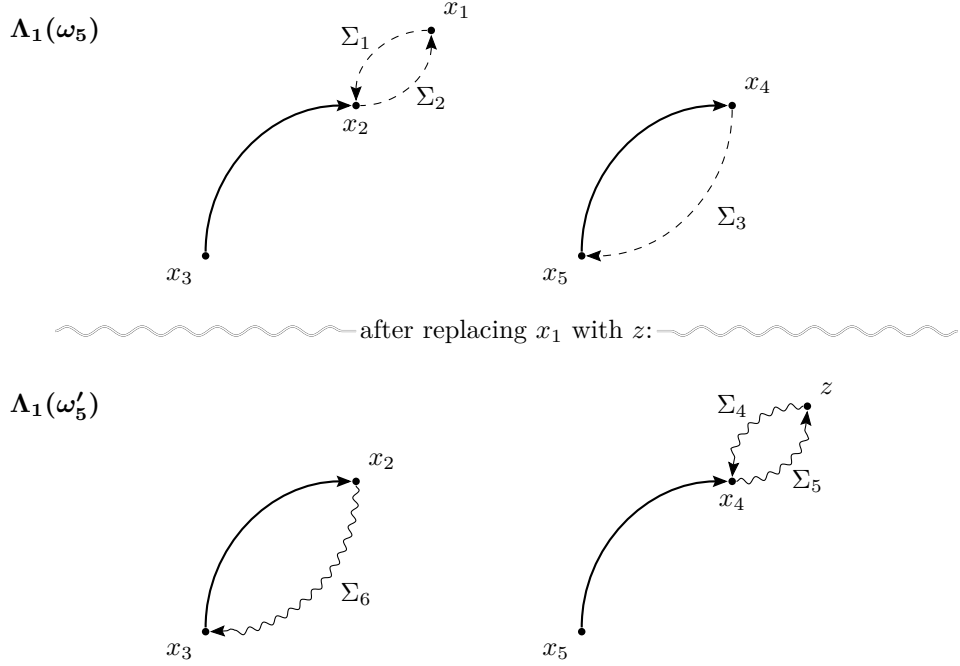
\begin{figure}[t]
\centering
\begin{tikzpicture}
    \def\arr{Stealth[length=2mm,inset=0.3mm]}
    \node[circle , fill , inner sep=1pt , label=above right:$x_1$] (a) at (1 , 3)  {};
    \node[circle , fill , inner sep=1pt , label=below:$x_2$] (b) at (0  , 2)  {};
    \node[circle , fill , inner sep=1pt , label=below left:$x_3$] (c) at (-2 , 0)  {};

    \node[circle , fill , inner sep=1pt , label=above right:$x_4$] (d) at (5  , 2)  {};
    \node[circle , fill , inner sep=1pt , label=below left:$x_5$] (e) at (3  , 0)  {};

    \node at (-4 , 3) {$\bs{\Lambda_1(\omega_5)}$};
    \node at (0 , 2.9) {$ \Sigma_1 $};
    \node at (1 , 2.1) {$ \Sigma_2 $};

    \node at (5   , 0.5) {$ \Sigma_3 $};

    \draw[arrows={\arr-} , black , dashed] (a) to [out=-90 , in=0] (b);
    \draw[arrows={\arr-} , black , dashed] {(b) to [out=90 , in=180] (a)};
    \draw[arrows={-\arr} , black , thick ] (c) to [out=90  , in=180] (b);

    \draw[arrows={-\arr}     , black, thick ] (e) to [out=90 , in=180] (d);
    \draw[arrows={-\arr},  dashed] (d) to[out=-90, in=0] (e);

    \draw [black!80, double, ultra thin, decorate,decoration={snake,amplitude=.6mm,segment length=5mm}] (-4,-1) -- (0,-1);
    \node at (2,-1)  {after replacing $ x_1 $ with $ z $:};
    \draw [black!80, double, ultra thin, decorate,decoration={snake,amplitude=.6mm,segment length=5mm}] (8,-1) -- (4,-1);

    %
    %

    \node[circle , fill , inner sep=1pt , label=above right:$x_2$] (b_lower) at (0 , -3)  {};
    \node[circle , fill , inner sep=1pt , label=below left:$x_3$] (c_lower) at (-2 , -5)  {};
                                                                                     
    \node[circle , fill , inner sep=1pt , label=above right:$z$] (z_lower) at   (6 , -2)  {};
    \node[circle , fill , inner sep=1pt , label=below:$x_4$] (d_lower) at (5       , -3)  {};
    \node[circle , fill , inner sep=1pt , label=below left:$x_5$] (e_lower) at (3  , -5)  {};

    \node at (-4 , -2) {$\bs{\Lambda_1(\omega_5')}$};
    \node at (0   ,-4.5) {$ \Sigma_6 $};

    \node at (5 ,-2) {$ \Sigma_4 $};
    \node at (6 ,-3) {$ \Sigma_5 $};

    \draw[arrows={-\arr} , black , thick ] (c_lower) to [out=90  , in=180] (b_lower);
    \draw[arrows={-\arr}, black      , solid, decorate,decoration={snake,amplitude=.4mm,segment length=2mm,post length=2mm}] (b_lower) to[out=-90, in=0] (c_lower);

    \draw[arrows={-\arr}     , black , solid, decorate,decoration={snake,amplitude=.4mm,segment length=2mm,post length=2mm}] (z_lower) to [out=180 , in=90] (d_lower);
    \draw[arrows={-\arr}     , black , solid, decorate,decoration={snake,amplitude=.4mm,segment length=2mm,post length=2mm}] {(d_lower) to [out=0 , in=-90] (z_lower)};
    \draw[arrows={-\arr}     , black, thick ] (e_lower) to [out=90 , in=180] (d_lower);
\end{tikzpicture}
\caption{Elements of $ \Lambda_1(\omega_5) $ compared to those in $ \Lambda_1(\omega_5') $. Arrow from node $ x $ to node $ y $ means that pair $ (x,y) $ is present in the respective adjacency graph. Solid arrows show pairs present in both graphs; wavy arrows represent pairs appearing only in $ \Lambda_1(\omega_5') $; dashed arrows those appearing only in $ \Lambda_1(\omega_5) $. For this small example, each of $ \Sigma_m $ contains exactly one term/edge.}
\label{fig:sums}
\end{figure}
To summarize, there holds
\[
E^k_s(\omega_N'; w) - E^k_s(\omega_N; w) = \Sigma_4 + \Sigma_5 + \Sigma_6 - \Sigma_1 - \Sigma_2 - \Sigma_3.
\]

As an illustration, all the six sums $ \Sigma_m $ are present when point $ x_1 $ is replaced with $ z $ in the tuple $ \omega_5 = (x_i)_1^5 $, shown in Figure~\ref{fig:sums}. In this figure, ordered pairs $ (x,y)\in \Lambda_1 $ for either tuple are represented as directed edges of a graph. 

To finish the proof, we will need an upper bound on $ \Sigma_6 -\Sigma_2 $. To that end, note that each pair in $ \Sigma_6 $, that is, 
$$ (x,y)\in \Lambda_k(\omega_N')\setminus\Lambda_k(\omega_N) \quad \text{such that } x,y \neq z, $$ 
must be replacing a pair having the form $ (x,x_1)$ in $ \Lambda_k(\omega_N) $, to keep the total number of outgoing edges from $ x $ equal to $ k $. Grouping the new pairs in $ \Sigma_6 $ with the removed ones in $ \Sigma_2 $ by their starting node gives
\[
    \Sigma_6 -\Sigma_2 = \sum_{\substack{x:\ (x,x_1)\in \Lambda_k(\omega_N),\\ z\notin\nn_k(x;\,\omega_N') }} \left( \frac{w(x , (x;\omega_N)_{k+1} ) }{\| x - (x;\omega_N)_{k+1} \|^{s}} - \frac{w(x , x_1) }{\| x - x_1\|^{s}} \right) \leq 0,
\]
since $ \|x-x_1\| \leq  \|x-(x;\omega_N)_{k+1}\| $, and $ w $ is marginally radial, so the expression $ w(x,y)/\|x-y\|^{-s} $ is nonincreasing with the distance $ \|x-y\| $ for every fixed $ x $. We used the notation $ (x,\omega_N)_{k+1} $ for the $ (k+1) $-st nearest neighbor to $ x $ in $ \omega_N $, as introduced in Section~\ref{sec:outline}. Finally, equation \eqref{eq:almost_min} implies
\begin{equation*}
    \begin{aligned}
        -(R+L) N^{s/d}
        &\leq \Sigma_4 + \Sigma_5 + \Sigma_6 - \Sigma_1 - \Sigma_2 - \Sigma_3 \\
        &\leq \Sigma_4 + \Sigma_5 + (\Sigma_6 - \Sigma_2) - \Sigma_1 \\
        &\leq  
        \sum_{x\in\nn_k(z;\omega_N')} \frac{w(z , x) }{\| z - x\|^{s}} +
        \sum_{x:z\in\nn_k(x;\omega_N')} \frac{w(x,z)}{ \| x - z\|^{s}} 
        - \frac{w(x_1 , x_2)}{\| x_1 - x_2\|^{s}}
        \\
        &\leq  \left(k M_w \cf^{-s} + n(k,p) M_w \cf^{-s} -m_w c_N^{-s}\right) N^{s/d},
    \end{aligned}
\end{equation*}
    where in the fourth inequality Lemma~\ref{lem:few_nns} is used to estimate the number of terms in the second sum. 

    This implies that whenever $ c_N < \delta N^{1/d} $, there holds
    \begin{equation}
        \label{eq:replace_point}
        c_N\geq \left(\frac{m_w}{(R+L)\,\cf^s + (k  + n(k,p)) M_w }\right)^{1/s} \cf,
    \end{equation}
    as desired.
\end{proof}
\begin{corollary}
    \label{cor:almost_minimizer_separation}
    Let $ w\equiv 1 $ and $ p=d $; let also $ \omega_N^* $ be such that $ E_s^k(\omega_N^*) \leq \mathcal E_s^k(A,N) + 1 $. Then equation~\eqref{eq:replace_point} implies 
    \[
        \Delta(\omega_N^*) \geq  \left(\frac{1}{1 + k  + n(k,d) }\right)^{1/s} \cf\, N^{-1/d}, \qquad N \geq N_0(d, \mathcal L_d(A)),
    \]
    where $ \cf = c(d) \mathcal L_d(A)^{1/d} $, as in the discussion after Proposition~\ref{prop:frostman}.
\end{corollary}
\begin{proof}
    It suffices to note that under such assumptions on $ \omega_N^* $, $ L =0 $, $ R\leq N^{-s/d} $, and $ m_w = M_w $ = 1 in equation~\eqref{eq:replace_point}.
\end{proof}
\begin{corollary}
    \label{cor:covering}
    The proof of Theorem~\ref{thm:separation} shows that there holds an optimal covering result, at least for some sublevel set of $ V+w $. In particular, when $ V \equiv 0 $, $ w\equiv 1 $ one has an optimal covering result:
    For a compact set $ A \subset \mathbb R^d $ with $ 0 < \mathcal \h_d(A) < \infty $, and a sequence of configurations $ \{ \omega_N^* \}_1^\infty $, such that 
    \[
        E_s^k(\omega_N) \leq \mathcal E_s^k(A, N) + RN^{s/d}, \qquad N\geq 1,
    \]
    for every $ y \in A $ there holds
    \[
        \dist(y,\omega_N) \leq C(s,k,d,A,R)\, N^{-1/d}.
    \]
\end{corollary}
\begin{proof}
    It suffices to note that in the proof of the theorem, one obtains the inequality~\eqref{eq:replace_point} between the covering radius at $ z \in D $, equal to $ \cf N^{-1/d} $, and the minimal separation, equal to  $ c_N N^{-1/d} $. Using $ \h_d(A) < \infty $ and a standard volume argument, one easily obtains an upper bound of $ C(A)N^{-1/d} $ on the optimal separation, at least for $ p=d $. In the case $ p>d $, one uses instead the finiteness of the Minkowski content, $ \m_d(A) < \infty $, to the same effect. In the sequel, we shall only need the optimal covering property for the unit cube in $ \mathbb R^d $ in Section~\ref{sec:cube}. It is not hard to see that the optimal covering holds on the $ L_1 $-sublevel set for $ V $, with $ L_1 $ from the statement of Theorem~\ref{thm:k_asympt}, and is not guaranteed on any larger sublevel set.
\end{proof}

\medskip 

\noindent{\bf Existence of minimizers of $ E^k_s $} The above theorem concerns configurations with near-optimal value of energy. A natural question to ask is, under which assumptions on $ w $ the functional \eqref{eq:k_energy} attains its minimum on $ A^N $; that is, whether $ E^k_s $ is lower semicontinuous. Recall that the topology on $ A^N $ is the product topology induced by the restriction of Euclidean metric to $ A $. In the following proof it will be convenient to use $ l^\infty $ norm on $ A^N $, so that distance between two configurations is
\[
    \rho(\omega_N', \omega_N'') = \max_i \|x_i' - x_i''\|.
\]
\begin{lemma}
    \label{lem:lower_semicont}
    Let $ V $ be lower semicontinuous on $ A $. If $ w $ is a weight  of the form
    \begin{equation*}
        w(x,y) = W(x, \|x-y\|),
    \end{equation*}
    with $ W $ lower semicontinuous on $A\times [0,\diam A] $, then $ E^k_s(\omega_N; w, V) $ is lower semicontinuous on $ A^N $ for any fixed $ N \geq 1 $.
\end{lemma}
\begin{remark}
    To see that $ w $ must indeed only depend on the distance $ \|x-y\| $, let $ d=2 $ and 
    $$ \omega_3^{(n)} = (x_1, x_2^{(n)}, x_3) =  \left( (0,0),\ \left(0,1+2^{-n}\right),\ (1,0) \right) $$
    be a sequence of 3-point configurations, converging to $ \omega_3 = ( x_i )_1^3 := \left( (0,0),\  (0, 1),\ (1,0)\right) $. Let further $ w $ be continuous, symmetric, and such that 
    $$ w((0,0), \ (0,1)) =3, \qquad w((0,0), \ (1,0)) =1. $$
    Then 
    \[
        E_1^1(\omega_3^{(n)}) = 3\|x_2^{(n)}\|^{-1} + 2\|x_3\|^{-1} \to 5, \quad n \to \infty,
    \]
    while 
    \[
        E_1^1(\omega_3) = 2\cdot3\|x_2\|^{-1} + \|x_3\|^{-1} =  7,
    \]
    since the tie-breaking convention in $ E_1^1 $ prefers points with smaller indices (see page \pageref{pg:ties}), thereby violating the lower semicontinuity.
\end{remark}
\begin{proof}[Proof of Lemma~\ref{lem:lower_semicont}]
    Fix a configuration $ \omega_N^\circ = \{ x_1^\circ,\ldots,x_N^\circ \} $. In this proof, points and indices related to $ \omega_N^\circ $ will be denoted by the $ ^\circ $ superscript. Objects related to a variable configuration $ \omega_N $, approaching $ \omega_N^\circ $ in the product topology, will not carry this superscript.

    If $ x_i^\circ = x_j^\circ $ for some $ i\neq j $,  $ E^k_s(\omega_N^\circ; w,V) = +\infty $. Due to the lower semicontinuity and nonnegativity of $ \|\cdot \|^{-s} $ and $ w $, and lower semicontinuity of $ V $, there holds
    \[ 
        E^k_s(\omega_N; w,V) \to +\infty, \qquad \hbox{ whenever }  \omega_N \to \omega_N^\circ \text{ in }  A^N.
    \]

    Let now $ \omega_N^\circ $ consist of distinct points and $ V\equiv 0 $. Fix an $ \epsilon > 0 $.
    Note that when $ \omega_N = \{ x_i \}_1^N $ is sufficiently close to $ \omega_N^\circ $ in the $ l^\infty $ metric on $ A^N $, so that for $\Delta $ as in \eqref{eq:separation},
    \[
        \| x_i - x_i^\circ \| < \Delta(\omega_N^\circ) /3, \qquad 1 \leq i \leq N,
    \]
    then the value of $ E_s^k(\omega_N) $ continuously depends on $ x_i $. In addition, by lower semicontinuity of $ w $, for a sufficiently small $ \delta_1 = \delta(\epsilon) $, one has
    \[
        w(x_i,x_j)\| x_i - x_j \|^{-s} \geq w(x_i^\circ,x_j^\circ)\| x_i^\circ - x_j^\circ \|^{-s} - \frac{\epsilon}{k}
    \]
    whenever $ r_{ij} := \| x_i - x_j \|  $ and $ r_{ij}^\circ := \| x_i^\circ - x_j^\circ \|  $ differ by at most $ \delta_1 $ for $ 1\leq i,j \leq N $ (we used here the specific form of $ w $).

    We shall further need to show that the nearest neighbor structure $ \Lambda_k(\omega_N) $ does not change much in a neighborhood of $ \omega_N^\circ $ -- or more precisely, distances to the $ k $ nearest neighbors remain approximately the same, even if the points themselves may be different.
    Fix an index $ i $. To obtain lower semicontinuity of $ E_s^k(\omega_N;w) $, it suffices to verify the semicontinuity at $ \omega_N^\circ $ only for the sum
    \[
        \sum_{x_j\in \nn_k(x_i;\omega_N)} \|x_i - x_j\| ^{-s}
    \]
    as a function of configuration $ \omega_N $.
    Consider distances from $ x_i^\circ $ to the other entries of $ \omega_N^\circ $:
    \[
        d_l := \| x_i^\circ - (x_i^\circ ; \omega_N^\circ)_l\|, \qquad  1\leq l \leq N.
    \]
    By the definition of $ (x_i^\circ; \omega_N^\circ)_l $, $ d_{l+1}\geq d_l $.
    Let $ \{ D_l \}_1^K $ be the strictly increasing sequence of unique values among $ \{ d_l \} $, $K \leq N$. Partition the multiset of entries of $ \omega_N^\circ $ as
    \[
        \bigsqcup_{l=1}^{K} J^\circ_{l},\qquad J_{l}^\circ:= \{ y\in \omega_N^\circ : \|x_i^\circ-y\| = D_l \},
    \]
    according to the unique distances to $ x_i^\circ $. Note that when a configuration $ \omega_N = (x_1,\ldots,x_N) \in A^N $ is such that $ \| x_j - x_j^\circ \| < \delta $, $ 1\leq j \leq N $, with
    \[
        \delta < \min \left[\delta_1,\ \min\left\{ (D_{l+1}-D_l)/4 \right\}_1^{K-1} \right],
    \] 
    there holds
    \[
        \| x_i - x_{j_1}\| < \| x_i - x_{j_2}\|, \qquad x_{j_1}^\circ \in J_{l_1}^\circ,\ x_{j_2}^\circ \in J_{l_2}^\circ, \ l_1 < l_2.
    \]
    That is, the entries of $ \omega_N\setminus\{ x_i \} $ can be collected into $ K $ groups, $ \{ J_l \}_1^K $, with any element of a group $ J_{l_1} $ closer to $ x_i $ than any element of another group $ J_{l_2} $ when $ l_1 < l_2 $. By construction,
    \[
        \# J_l^\circ = \# J_l,\qquad 1\leq l \leq K,
    \]
    so there is a bijection between elements of corresponding groups.
    Observe also that for every pair of elements  $ y,z \in J_{l} $ and a pair $ y^\circ, z^\circ \in J_l^\circ $, belonging to corresponding groups, there holds
    \[
        \big|\, |\| y- z \| - \| y^\circ - z^\circ \|\,\big| < 2\delta,
    \]
    which by the definition of $ \delta_1 $ gives
    \begin{equation*}
        w(y,z)\| y - z \|^{-s} \geq w(y^\circ,z^\circ)\| y^\circ - z^\circ \|^{-s} - \frac{\epsilon}{k}.
    \end{equation*}
    By the aforementioned bijection between groups $ J_l^\circ $ and $ J_l $, this implies
    \[
    \sum_{x_j\in \nn_k(x_i;\omega_N)} w(x_i,x_j)\| x_i - x_j \|^{-s} \geq \sum_{x_j\in \nn_k(x_i^\circ;\omega_N^\circ)} w(x_i^\circ,x_j^\circ)\| x_i^\circ - x_j^\circ \|^{-s} - {\epsilon}, \]
    proving the lower semicontinuity of the point energy for a single $ x_i $, and therefore lower semicontinuity of $ E_s^k(\cdot\,; w) = E_s^k(\cdot\,; w, 0) $. Addition of an external field is only introducing another lower semicontinuous term, which completes the proof.
\end{proof}
\begin{corollary}
    Under the assumptions of Lemma~\ref{lem:lower_semicont}, minimizing the $ k $-energy for any $ k \geq 1 $ and $ s>0 $ yields a separated configuration.
\end{corollary}

\section{Proofs of the main results}
\label{sec:proofs}
\subsection{Asymptotics on cubes}
\label{sec:cube} 
Let us demonstrate the existence of asymptotics on cubes for the Riesz $ k $-energy functionals, following the outline in Section~\ref{sec:outline}. We shall need this fact only for the unweighted $ E_s^k $ -- or, equivalently, for the constant weight $ w $ and zero external field $ V $. In this section, the ambient space is $ \mathbb R^d $.
\begin{lemma}
    \label{thm:cube}
  For $ s> 0 $ and  $ k ,d$ positive integers, the following limit exists, and is positive and finite:
 \begin{equation}\label{CsdkDef}
     C_{s,d}^k:=   \lim_{N\to\infty}  \,\frac{\mathcal E_s^k(\q_d,N)}{N^{1+s/d}},
\end{equation}
 which is the constant appearing in \eqref{CsdkDef0}.
\end{lemma}
\begin{proof}
  Set
    \[
        \underline \l := \liminf_{N\to\infty} \frac{\mathcal E_s^k(\q_d, N)}{N^{1+s/d}}.
    \] 
    Fix $\epsilon>0$ and let  $ \{ \underline \omega_n \}_{n\in \uline{\mathscr N}} $ be a subsequence of $n$-point configurations in $\q_d$  for which 
    $$E_s^k(\underline \omega_n) < (\underline \l+\epsilon)n^{1+s/d} , \qquad n\in \uline{\mathscr N}.$$
  We shall show that
    \[
        \overline \l := \limsup_{N\to\infty} \frac{\mathcal E_s^k(\q_d, N)}{N^{1+s/d}},
    \] 
    equals $\underline \l$, establishing the existence of the limit \eqref{CsdkDef}.

  For any $N\geq 1$, let  $ \{\overline \omega_N \} $ be an $N$-point configuration in $\q_d$ such that
    \begin{equation}\label{Eskomegaup}E_s^k(\overline\omega_N) < \mathcal E_s^k(\q_d, N) + 1.
    \end{equation}
  Let    $n\in\uline{\mathscr N}$ and  $L$   be the unique positive integer  such that 
    \begin{equation}
        \label{eq:squeeze}
        n(L-1)^d  \leq N < nL^d.
    \end{equation}
    By Corollary~\ref{cor:covering}, there is a positive constant $c(s,k,d)$ such that the distance to the $ k $-th nearest neighbor in $ \underline \omega_n $ is at most $ c(s,k,d) n^{-1/d} $.  Let
    \[
        \gamma := 1- c(s,k,d) n^{-1/d},    \]
 and consider the  following configuration obtained by tiling $ \q_d $ with copies of $\gamma \underline \omega_n / L $:
    \[
        \omega := \bigcup_{
        \boldsymbol i\in\, (L\mathbb Z)^d }
        \left(\frac\gamma{L}\underline \omega_n +  \frac{\boldsymbol i}L\right),
    \]
    where $ L\mathbb Z := \{ 0,1,\ldots, L-1 \} $.  We observe that with this choice of $\gamma$    the $ k $ nearest neighbors in $ \omega $ for every point in the subset $ \frac\gamma{L}\underline \omega_n +  \frac{\boldsymbol i}L $ also belong to this subset. 
    
    We next  obtain an upper bound for $ E_s^k(\overline\omega_N) $ using $\omega$.
       By the scale invariance of $ E_s^k $ there holds
    \[
        E_s^k\left(\frac\gamma{L}\underline \omega_n\right) =  \left(\frac L\gamma\right)^s E_s^k(\underline\omega_n),
    \]
    and so from \eqref{Eskomegaup} and \eqref{eq:squeeze}, we have 
 \begin{equation}\label{Eskupperbnd}
        \begin{aligned} 
             \frac{E_s^k(\overline\omega_N)}{N^{1+s/d}}
            &\leq \frac{E_s^k(\omega)+1}{N^{1+s/d}} = \frac{L^d\, \left(\frac L\gamma\right)^s E_s^k(\underline\omega_n) +1}{N^{1+s/d}} \\
            &
            \leq 
            \frac{L^d\, \left(\frac L\gamma\right)^s E_s^k(\underline\omega_n)}{(n(L-1)^d)^{1+s/d}} + N^{-1-s/d} \\
            &\leq \gamma^{-s} \left(\frac{L}{L-1}\right)^{s+d}(\underline \l+\epsilon)+N^{-1-s/d},
    \end{aligned}
\end{equation}
    where in the first equality we used that due to the choice of $ \gamma $, no interactions between different tiles enter the sum for $ E_s^k(\omega) $.  Taking the limit superior  as $N\to \infty$ in \eqref{Eskupperbnd} for fixed $n$ which implies   $L\to \infty$, gives
    $$\overline \l\le    \gamma^{-s}(\underline \l+\epsilon).$$
  Then taking $n\to \infty$, $n\in \uline{\mathscr N}$, (in which case $\gamma\to 1$) shows that  
    $ \overline \l\le    \underline \l+\epsilon .$  Since $\epsilon$ is arbitrary, $ \overline \l\le    \underline \l$ and so the limit    $C_{s,d}^k$ in \eqref{CsdkDef}   exists in $[0,\infty]$.

    It remains to show that $C_{s,d}^k$ is finite and positive.  
Notice that   ${\mathcal E}_s^k(\q_d,N)= O(N^{1+s/d} )$ follows by placing the points in $ \omega_N $ in the vertices of the cubic lattice which shows that $C_{s,d}^k$ is finite.  
To see that  $C_{s,d}^k$ is positive, observe that 
for any configuration $ \omega_N \in ([0,1]^d)^N $ of $ N $ distinct points there holds
\[
    E_s^k(\omega_N) = \sum_{i=1}^N \sum_{y\in \nn_k(x_i; \omega_N)} \|x_i-y\|^{-s} \geq  \sum_{i=1}^N \|x_i - (x_i;\omega_N)_1\|^{-s} =:\sum_{i=1}^N r_i^{-s},
\]
where as before we write $ (x_i;\omega_N)_1 $ for the nearest neighbor of $ x_i $. Notice that the interiors of balls $ B(x_i, r_i/2) $ are disjoint and contained in $ [-(1+\sqrt d)/2,\ (1+\sqrt d)/2]^d $, so $ v_d\sum_i r_i^d \leq 2^d (1+\sqrt d)^d $, where $ v_d $ denotes the volume of $ d $-dimensional unit ball. In conjunction with Jensen's inequality this implies
\[
    \sum_{i=1}^N r_i^{-s} = \sum_{i=1}^N (r_i^d)^{-s/d} \geq N \left(\frac1N\sum_{i=1}^N r_i^d\right) ^{-s/d} \geq N^{1+s/d}\, \left(2^d (1+\sqrt d)^d / v_d\right)^{-s/d},
\]
which is the desired lower bound.
\end{proof}
From scale- and translation-invariance of $ E_s^k $, we obtain also the asymptotics for a general cube in $ \mathbb R^d $.
\begin{corollary}
    \label{cor:cube_general}
  For $ s> 0 $,  $ k,d$ positive integers, and a cube $ Q = x +a\q_d \subset \mathbb R^d $ the following limit exists:
 \begin{equation*}
     \lim_{N\to\infty}  \,\frac{\mathcal E_s^k(Q,N)}{N^{1+s/d}} = \frac{C_{s,d}^k}{a^s} = \frac{C_{s,d}^k}{\h_d(Q)^{s/d}}.
\end{equation*}
\end{corollary}
The lower bound on $ \mathcal E_s^k $ derived in the proof of Lemma~\ref{thm:cube} relied on the condition $ A = [-1/2,1/2]^d $. To verify the short-range property, it will be necessary that  $ \mathcal E_s^k(A,N) $ grow to infinity with $ N $, for which it suffices to assume that $ A $ is compact. In the following lemma we establish such growth for energy on compact sets.
\begin{lemma}
    \label{lem:lower_bound_nonsmooth}
    Suppose $ A\subset \mathbb R^d $ is a compact set, $ k\geq 1 $, and $s>0$. Then
    \[
        \liminf_{N\to\infty} \frac{\mathcal E_s^k(A,N)}{N^{1+s/d}} \geq C(s,d)\,(\mathcal H_d(A))^{-s/d}.
    \]
\end{lemma}
\begin{proof}
    Clearly, it is sufficient to assume $ k=1 $.
    Fix an $ \epsilon > 0 $. By Besicovitch's covering theorem~\cite[Theorem 2.7]{mattila1995geometry}, for every sufficiently small $ r>0 $ there exist a collection of balls $ \{ B_m \}_1^M := \{ B(x_m,r) \}_{m=1}^M $ that cover $ A $ and satisfy
    \[
        \sum_{m=1}^M v_d r^d = Mv_dr^d \leq c(d) \mathcal H_d (A_r),
    \]
    where as usual, $ A_r $ is the $ r $-neighborhood of $ A $.
    In fact, each point of $ \mathbb R^d $ is contained in at most $ c(d) $ balls among $ \{ B_m \}_{m=1}^M $. 
    Let $ \omega_N \in A^N $  be an arbitrary configuration of $ N $ points and 
    denote by $ \omega_N' $ the subconfiguration of its elements $ x $, contained in some $ B_m $ that also contains at least one other element of $ \omega_N $. Then $ \#(\omega_N \setminus \omega_N') \leq M $.
    In addition, whenever $ B_m $ contains at least $ 2 $ elements of $ \omega_N $, for every such element $ x\in \omega_N' \cap B_m $ there holds
    \[ 
        r_x = \|x - (x;\omega_N)_1 \| \leq  2r.
    \]
    Note that the balls $ \{ B(x, r_x/2) \}_{x\in\omega_N} $ are disjoint, which gives
    \[
        \sum_{x\in B(x_m,r)}  v_d r_x^d \leq  v_d(4r)^d,
    \]
    Applying Jensen's inequality, one has
    \[
        \begin{aligned}
        E_s^k(\omega_N) 
        &\geq \sum_{x\in\omega_N} r_x^{-s} \geq \sum_{x\in\omega_N'} r_x^{-s} =  \sum_{x\in\omega_N'} (r_x^d)^{-s/d} \\
        &\geq \#\omega_N' \left(\frac1{\#\omega_N'} \sum_{x\in\omega_N'} r_x^d\right)^{-s/d} \geq (\#\omega_N')^{1+s/d} (4^d Mr^d)^{-s/d}\\
        &\geq (N-M)^{1+s/d} C(s,d) (\mathcal H_d(A_r))^{-s/d}.
        \end{aligned}
    \]
    Since $ M $ is fixed, this gives further
    \[
        \liminf_{N\to\infty} \frac{\mathcal E_s^k(A,N)}{N^{1+s/d}} \geq C(s,d)(\mathcal H_d(A_r))^{-s/d}.
    \]
    By taking $ r\downarrow0 $, the lemma follows.
\end{proof}
We conclude this section with the proof of the short-range property for the functional $ E_s^k $. The proof will use that $ \mathcal E_s^k(A,N) $ grows to infinity, as we just established for all compact subsets of $ \mathbb R^p $.
\begin{lemma}
    \label{lem:short_range}
    Let $ A_1, A_2 \subset \mathbb R^p $ be disjoint compact sets.   If $ \omegaNseq $ is a sequence of $N$-point configurations in  $ A_1\cup A_2 $ for $ N \geq 2 $,   then
    \begin{equation*}
        \lim_{N\to\infty} \frac{  E_s^k(\omega_N\cap A_1)+E_s^k(\omega_N\cap A_2)}{E_s^k(\omega_N)} = 1.
    \end{equation*}
\end{lemma}
\begin{proof}
    Notice that for any $ x\in \omega_N $,
    \[
        \|x- (x;\omega_N)_l\| \leq \|x- (x;\omega_N\cap A_m)_l\|, \qquad 1\leq l \leq k, \ m=1,2. 
    \]
    As a result, there holds
    \[
        \sum_{y\in\nn_k(x;\omega_N)} \|x-y\|^{-s} \geq \sum_{y\in\nn_k(x;\omega_N \cap A_m)} \|x-y\|^{-s}, \qquad x\in\omega_N, \ m=1,2,
    \]
    which gives
    \[
          E_s^k(\omega_N\cap A_1)+E_s^k(\omega_N\cap A_2) \leq E_s^k(\omega_N)
    \]
    and 
    \[
        \limsup_{N\to\infty} \frac{  E_s^k(\omega_N\cap A_1)+E_s^k(\omega_N\cap A_2)}{E_s^k(\omega_N)} \leq 1.
    \]

    It remains to derive the   estimate for the lower limit.
    Since $ A_1, A_2 $ are compact and disjoint, 
    $$ h:= \dist(A_1,A_2) > 0. $$
    Pick an element $ x\in \omega_N $; without loss of generality, $ x\in A_1 $. There are two possibilities: i) $ \nn_k(x;\omega_N) \subset A_1 $, in which case all the terms of the form
    \[
        \sum_{y\in\nn_k(x;\omega_N)} \|x-y\|^{-s}
    \]
    are shared by the two sums $ E_s^k(\omega_N\cap A_1)+E_s^k(\omega_N\cap A_2) $ and $ E_s^k(\omega_N) $; ii) $ \nn_k(x;\omega_N) \cap A_2\neq \emptyset$, in which case  (see Section~\ref{sec:outline} for the adjacency graph notation)
    \[
        \Lambda_k(\omega_N) \setminus [\Lambda_k(\omega_N\cap A_1) \cup \Lambda_k(\omega_N\cap A_2)] \supset \{ (x,y) : y\in \nn_k(x;\omega_N), \ y\in A_2 \};
    \]
    in other words, some of the edges connecting $ x $ to its nearest neighbors in $ \omega_N $ are missing from the union of adjacency graphs $ \bigcup_{m=1}^2 \Lambda_k(\omega_N\cap A_m) $. On the other hand, all the terms occurring in $ E_s^k(\omega_N) $ but not in $ E_s^k(\omega_N\cap A_1)+E_s^k(\omega_N\cap A_2) $ are precisely of this form, so collecting all the pairs with $ x\in A_1 $ into
    \[
        G_1 : = \bigcup_{x\in\omega_N\cap A_1} \{ (x,y) : y\in \nn_k(x;\omega_N), \ y\in A_2 \}
    \]
    and those with $ x\in A_2 $ into
    \[
        G_2 : = \bigcup_{x\in\omega_N\cap A_2} \{ (x,y) : y\in \nn_k(x;\omega_N), \ y\in A_1 \},
    \]
    we conclude
    \[
        \Lambda_k(\omega_N) \setminus [\Lambda_k(\omega_N\cap A_1) \cup \Lambda_k(\omega_N\cap A_2)] = G_1\cup G_2.
    \]
    It follows
    \[
        \begin{aligned}
        E_s^k(\omega_N) 
        &= \sum_{(x,y) \in \Lambda_k(\omega_N)} \|x-y\|^{-s} \\
        &\leq \left(
        \sum_{(x,y) \in \Lambda_k(\omega_N\cap A_1)} + \sum_{(x,y) \in \Lambda_k(\omega_N\cap A_2)} + \sum_{(x,y) \in G_1} + \sum_{(x,y) \in G_2 }\right) \|x-y\|^{-s}\\
        &\leq E_s^k(\omega_N \cap A_1) + E_s^k(\omega_N \cap A_2) + Nk\,h^{-s}.
        \end{aligned}
    \]
    Here in the last inequality we used that $ \|x-y\|^{-s}  \leq h^{-s} $ for $ x,y $ placed in different $ A_m $, and that the total number of edges in $ \Lambda_k(\omega_N) $ is $ Nk $. Using Lemma~\ref{lem:lower_bound_nonsmooth} with $ p=d $, we conclude that $ Nk\,h^{-s} = o(E_s^k(\omega_N)) $, \(N\to\infty\), whence dividing the last display through by $ E_s^k(\omega_N) $ and taking $ N $ to infinity yields 
    \[
        \liminf_{N\to\infty} \frac{E_s^k(\omega_N \cap A_1) + E_s^k(\omega_N \cap A_2)}{E_s^k(\omega_N)} \geq 1,
    \]
    completing the proof of the lemma.
\end{proof}

\subsection{Stability and poppy-seed bagel asymptotics for \texorpdfstring{$ E_s^k $}{Esk}}
In this section we establish a set-stability result for $ E^k_s $ as described in \eqref{eq:set_conts}.   This lemma will play a key role in the proof of Theorem~\ref{thm:poppy_seed_k} which is the main goal of this section.

\begin{lemma}
    \label{lem:stable}
    For a compact set $ A \subset \mathbb R^{p} $ with $ 0 < \mathcal M_d(A) < \infty  $, $ s> 0 $, and any $ \epsilon > 0 $, there exists a $ \delta = \delta(\epsilon, s,k,p,d,A) > 0 $ such that the inequalities
    \[
        \liminf_{N\to \infty}\frac{\mathcal E^k_s(A,N)}{N^{1+s/d}} \geq (1-\epsilon)
        \liminf_{N\to \infty} \frac {\mathcal E^k_s(D,N)}{N^{1+s/d}}, \qquad 
        \limsup_{N\to \infty}\frac{\mathcal E^k_s(A,N)}{N^{1+s/d}} \geq (1-\epsilon)
        \limsup_{N\to \infty} \frac {\mathcal E^k_s(D,N)}{N^{1+s/d}}
    \] 
    hold whenever a compact set $  D \subset A $ satisfies $ \m_d (D) > (1-\delta) \m_d (A)  $. In the case $ d=p $, $ \delta $ can be chosen independently of $ A $.
\end{lemma}
\begin{proof}
    Let $ \omega_N^* $ be a sequence of configurations satisfying $ E_s^k(\omega_N^*) < \mathcal E_s^k(A,N) + 1 $, $ N\geq 1 $. According to Theorem~\ref{thm:separation}, the separation for this sequence satisfies $ \Delta(\omega_N^*) \geq C N^{-1/d} $ for $ C = C(s,k,d,A) $.     

    The proof will consist in demonstrating a way to retract configurations from $ A $ to $ D $ without increasing the value of $ E^k_s $ on them too much. 
    We will first show that most of $ x_i\in\omega_N^* $ have a point from $ D $ close to them, for $ N $ sufficiently large. In this proof let us write $ S(r) = S_r $ for the closed $ r $-neighborhood of a set $ S $; observe that $ D(r) \subset A(r) $ for any $ r > 0 $.   

    Let $ \delta \in (0,1) $ and $ D \subset A $ be a compact set satisfying $ \m_d (D) > (1-\delta) \m_d (A) $. Then for all sufficiently small $ r>0 $, $ \mathcal L_p\left[A(r)\right] - \mathcal L_p\left[D(r)\right] < v_{{p}-d} r^{p-d}\, 3\delta \m_d(A) $, by the definition of Minkowski content. In particular, for
    $ N \geq N_0 =  N_0(A,D,\delta)$ and any $ 0< \gamma < C/4 $ there holds
    \[
        \mathcal L_p\left[A(\gamma N^{-1/d})\right] - \mathcal L_p\left[D(\gamma N^{-1/d})\right] < v_{{p}-d} \gamma^{{p}-d} N^{-({p}-d)/d}\, 3\delta \m_d(A).
    \]
    Thus the number of disjoint balls of radius $ \gamma N^{-1/d} $ that can be contained in $ A(\gamma N^{-1/d})\setminus D(\gamma N^{-1/d}) $ is at most 
    \[
        \frac{v_{{p}-d} \gamma^{{p}-d} N^{-({p}-d)/d}\, 3\delta \m_d(A)}{v_p \gamma^{p} N^{-{p}/d}} = \delta\, c({p},d) \gamma^{-d} N \m_d(A).
    \]
    It follows that for at least $ N ( 1- \delta c({p},d)\gamma^{-d}\m_d(A) )  $ points in $ \omega_N^* $, a closest point in $ D $ is at most distance
    \[
        2\gamma N^{-1/d}
    \]
    away. Consider the subconfiguration $  (x_i')_i \subset \omega_N^* $ for which this is the case, and for each $ x_i' $ find a closest point in $ D $. Denote the resulting set by $ \omega $. By the preceding discussion, it can be assumed
    \begin{equation}
        \label{eq:cardinality_bound}
        N_\omega := \#\omega = \left\lfloor N \left( 1- \delta c({p},d) \gamma^{-d}\m_d(A) \right) \right\rfloor, \qquad N \geq N_0. 
    \end{equation}
    Consider a pair $ x_i' $, $ x_j' $; let their nearest points in $ \omega $ be $ y_i $ and $y_j$ respectively. Since the separation between entries of $ \omega_N^* $ is at least $ CN^{-1/d} $, there holds
    \begin{equation*}
        \|y_i - y_j\| \geq (1-4\gamma/C) \| x_i' - x_j' \| > 0, \qquad i \neq j,
    \end{equation*}
    where we used that $ 4\gamma < C $.
    Due to the scaling properties of the kernel $ \|x-y\|^{-s} $, this implies in turn
    \[
        E^k_s(\omega_N^*) \geq (1-4\gamma/C)^{s} E^k_s(\omega).
    \]
    By the estimate \eqref{eq:cardinality_bound} on the cardinality $ N_\omega $, we have finally
    \begin{equation*}
        \frac{E^k_s(\omega_N^*)}{N^{1+s/d}} \geq {(1-4\gamma/C)^{s}}{\left( 1- \delta c({p},d) \gamma^{-d} \m_d(A) \right)^{1+s/d}}\cdot \frac{\mathcal E^k_s(D,N_\omega)}{N_\omega^{1+s/d}}, \qquad N\geq N_0.
    \end{equation*}
    Setting $ \gamma = C\delta^{1/2d} $ gives
    \begin{equation}
        \label{eq:projected_omega_n}
        \frac{\mathcal E^k_s(A,N)+1}{N^{1+s/d}} \geq {(1-4\delta^{1/2d})^{s}}{\left( 1-  \sqrt \delta\,c({p},d)  \m_d(A)/C^d \right)^{1+s/d}}\cdot \frac{\mathcal E^k_s(D,N_\omega)}{N_\omega^{1+s/d}}, \qquad N\geq N_0,
    \end{equation}
    implying the claim of the lemma for $ \liminf $ and a suitably small $ \delta $.

    To prove the claim for $ \limsup $, observe that following the above construction, given a cardinality $ N \geq N_0 $, one obtains cardinality
    \[
        N_\omega = \left\lfloor N\left(1- \sqrt \delta\,c({p},d)  \m_d(A)/C^d\right)\right\rfloor =: \lfloor N(1-c\sqrt\delta) \rfloor,
    \]
    for which inequality~\eqref{eq:projected_omega_n} holds. Here $ c = c(s,k,p,d,A) $. 
    Since the image of the set $\{ N : N \geq N_0 \}$ under the mapping 
    \[
        n \mapsto \lfloor n(1-c\sqrt\delta )\rfloor 
    \]
    contains $ \{ N : N\geq N_0 \} $ for $ \delta < (1/2c)^2 $, for every given $ N_\omega \geq N_0 $, there exists a cardinality $ N $, such that the inequality~\eqref{eq:projected_omega_n} holds with these particular values of $ N $ and $ N_\omega $.
    Now let $ \mathscr N \subset \mathbb N $ be a sequence along which 
    \[
        \limsup_{N\to \infty} \frac {\mathcal E^k_s(D,N)}{N^{1+s/d}}
    \]
    is attained; taking $ \mathscr N \ni N_\omega \to \infty $ in \eqref{eq:projected_omega_n} completes the proof for $ \limsup $.

    Finally, for $ d=p $, notice that by Corollary~\ref{cor:almost_minimizer_separation}, $ \m_d(A) / C^d = c(s,d,k) $ for sufficiently large $ N $, so indeed $ \delta $ in \eqref{eq:projected_omega_n} is independent of $ A $. This proves the last claim of the lemma.
\end{proof}
In the last auxiliary result before the main theorem of this section, we show that any functional equipped with the monotonicity, short-range, and stability properties from Section~\ref{sec:outline}, and for which the asymptotics are known for all compact subsets of $ \mathbb R^d $, also has asymptotics on $ (\h_d,d) $-rectifiable subsets of $ \mathbb R^p $ with $  \Hd(A) = \mathcal{M}_d(A) $, $ p\geq d $. In addition, the formula for the asymptotics coincides with that on the compact subsets of $\mathbb R^d $. The precise statement follows; in it, we say that  functionals $\e_N$, \(N\geq 1\), acting on collections $ \omega_N $, are {\it continuous under near-isometries} if for any $ \epsilon > 0 $, there is a $ \gamma $ such that for every bi-Lipschitz map $ \psi:\mathbb R^p \to \mathbb R^p $ with constant  $ (1+\gamma) $, we have
\[
    (1+\epsilon)^{-1}\, \e_N(\omega_N) \leq \e_N(\psi(\omega_N)) \leq (1+\epsilon) \,\e_N(\omega_N), \qquad N \geq 1.
\]
It is also understood that there is a fixed isometric embedding $ \mathbb R^d \subset \mathbb R^p $, $ d\leq p $, and \( \mathbb R^d \) is identified with its image under the embedding. In particular, functionals on collections of points $ \omega_N \in (\mathbb R^p)^N $ are automatically defined on collections $ \omega_N \in (\mathbb R^d)^N $.
\begin{lemma}
    \label{lem:federer}
    Suppose $ \e_N: (\mathbb R^p)^N \to [0,\infty] $, $ N\geq 1 $, is a sequence of continuous under near-isometries functionals, and a number $ s>0 $ is fixed. For a compact $ A \subset \mathbb R^p $, denote $ \e_N^*(A) :=\inf_{\omega_N \in A^N }\, \e_N(\omega_N) $. Assume that $ \e_N $ have the following properties:
    \begin{enumerate}[1)]
        \item\label{it:monot}
            $\e_N^* (A) \geq \e_N^* (B) $ whenever $ A\subset B \subset \mathbb R^p $ are compact sets.
        \item\label{it:shortrange}
            Suppose $ A_1, A_2 \subset \mathbb R^p $ are disjoint compact sets.   If $ \omegaNseq $ is a sequence of $N$-point configurations in  $ A_1\cup A_2 $, then 
            \begin{equation*}
                \lim_{N\to\infty} \frac{  \e_{N_1}(\omega_N\cap A_1)+\e_{N_2}(\omega_N\cap A_2)}{\e_N(\omega_N)} = 1,
            \end{equation*}
            where $ N_m $ are the cardinalities of the intersections $ \omega_N \cap A_m $, $ m=1,2 $.
        \item\label{it:stable}
            For every compact $ A\subset \mathbb R^p $ and $ \epsilon \in (0,1) $ there is a $\delta>0$ such that whenever a compact $ D \subset A $ satisfies $ \mathcal M_d(D) \geq (1 - \delta )\,\mathcal M_d(A)$, we have
            \begin{equation*}
                \liminf_{N\to \infty}\frac{\e_N^*(A)}{N^{1+s/d}} \geq (1-\epsilon)
                \liminf_{N\to \infty} \frac {\e_N^*(D)}{N^{1+s/d}}.
            \end{equation*}
        \item\label{it:Rdlimit}
            For every compact $ K \subset \mathbb R^d \subset \mathbb R^p $,
            \[
                \lim_{N\to \infty} \frac{\e_N^*(K)}{N^{1+s/d}} = \frac{C_\e}{\h_d(K)^{s/d}}.
            \]
    \end{enumerate}
    Then, for every $ (\h_d,d) $-rectifiable compact set $ A \subset \mathbb R^p $ with $  \Hd(A) = \mathcal{M}_d(A) > 0 $, one has
    \[
        \lim_{N\to \infty} \frac{\e_N^*(A)}{N^{1+s/d}} = \frac{C_\e}{\h_d(A)^{s/d}}.
    \]
\end{lemma}
\begin{proof}
    Fix an $ \epsilon > 0 $ and let $ \delta >0 $ be as in the stability assumption \ref{it:stable}. Without loss of generality, $ 0< \delta < \epsilon $. By a standard fact from geometric measure theory \cite[Lemma 3.2.18]{federerGeometric1996a}, there exist bi-Lipschitz maps $ \psi_m:\mathbb R^d\to \mathbb R^p $ with constant smaller then $ (1+\gamma) $, and compact sets $ K_m \subset \mathbb R^d $, $ 1\leq m \leq M $ such that sets $ \{ \psi_m(K_m) \} $ are disjoint, contained in $ A $, and
    \[
        \h_d\left(A\setminus \bigcup_{m=1}^M \psi (K_m) \right) < \delta \mathcal H_d(A).
    \]
    Without loss of generality, $  0< \gamma < \epsilon $.
    Denoting $ \tilde A:= \bigcup_m \psi (K_m) \subset A $, from the monotonicity assumption \ref{it:monot} we have
    \begin{equation}
        \label{eq:upper_A}
        \limsup_{N\to \infty} \frac{\e_N^*(A)}{N^{1+s/d}}  \leq \limsup_{N\to \infty} \frac{\e_N^*(\tilde A)}{N^{1+s/d}}.
    \end{equation}
    On the other hand, both $ A $ and $ \tilde A $ are compact, $ (\h_d,d) $-rectifiable, and $ \h_d(A) = \m_d(A) $, $ \h_d(\tilde A) = \m_d(\tilde A) $, see \cite[Lemma 4.3]{borodachovLow2014}; combined with the assumption on $ \delta $, this means the stability property \ref{it:stable} applies, so that
    \begin{equation}
        \label{eq:lower_A}
        \liminf_{N\to \infty} \frac{\e_N^*(A)}{N^{1+s/d}}  \geq (1-\epsilon) \liminf_{N\to \infty} \frac{\e_N^*(\tilde A)}{N^{1+s/d}}.
    \end{equation}
    By the last two displays, it suffices to derive the asymptotics of $ \e_N^* $ for $ \tilde A $. This is where the short-range assumption \ref{it:shortrange} comes into play. 

    Consider a sequence $ \omega_N^* \in (\tilde A)^{N} $, $ N\geq 1 $, such that $ \e_N(\omega_N^*) < \e_N^*(\tilde A) + 1 $. 
    Let $ A_m:= \psi_m(K_m) \subset \tilde A $ and $ N_m:= \#(\omega_N^*\cap A_m) $.
    By passing to a subsequence, the following limits can be assumed to exist:
    \[
        \beta_m := \lim_{N\to \infty} \frac{N_m }{N}, \qquad 1\leq m \leq M.
    \]
    Using assumption \ref{it:shortrange}, the fact that $ A_m $ are disjoint, and the choice of $ \psi_m $, we have
    \[
        \begin{aligned}
            \liminf_{N\to \infty} 
             \frac{\e_N^*(\tilde A) +1}{N^{1+s/d}} 
             &\geq \liminf_{N\to \infty} \frac{\e_N(\omega_N^*)}{\sum_m \e_{N_m}(\omega_N^* \cap A_m)}\cdot \frac{\sum_m \e_{N_m}(\omega_N^* \cap A_m)}{N^{1+s/d}}\\
             & = \liminf_{N\to \infty} \sum_m \left(\frac{N_m}{N}\right)^{1+s/d}\cdot \frac{\e_{N_m}(\omega_N^* \cap A_m)}{N_m^{1+s/d}}\\
             &\geq    \sum_{m} \beta_m^{1+s/d} \cdot \liminf_{N\to \infty}\frac{\e_{N_m}(\omega_N^* \cap A_m)}{N^{1+s/d}} \\
             &\geq  (1+\epsilon)^{-1} \sum_{m} \beta_m^{1+s/d} \cdot \liminf_{N\to \infty}\frac{\e_{N_m}(\psi_m^{-1}(\omega_N^*) \cap K_m)}{N^{1+s/d}},
    \end{aligned}
    \]
    where the last inequality used that $ \psi_m^{-1}(A_m) = K_m $ with a bi-Lipschitz constant smaller than $ (1+\epsilon)$. Using the asymptotics on subsets of $ \mathbb R^d $ in assumption \ref{it:Rdlimit}, we conclude further
    \begin{equation}
        \label{eq:lower_Atilde}
        \liminf_{N\to \infty} \frac{\mathcal \e_N^*(\tilde A)}{N^{1+s/d}} 
        \geq  (1+\epsilon)^{-1} \sum_{m} \beta_m^{1+s/d} \cdot \frac{C_\e}{\h_d(K_m)^{s/d}} \geq (1+\epsilon)^{-1} \frac{C_\e}{\mathcal H_d(\tilde A)^{s/d}}.
    \end{equation}
    The last inequality in \eqref{eq:lower_Atilde} follows by optimizing over nonnegative $ \{ \beta_m \} $ with $ \sum_m \beta_m =1 $. The minimum is achieved for $ \beta_m = \h_d(K_m)/\sum \h_d(K_m) $, \(1\leq m \leq M\).

    To obtain an upper bound for the asymptotics, we choose $ N_m$  in such a way that $ \sum_m N_m = N $, $ N_m/N \to  \h_d(K_m)/\sum_m\h_d(K_m)  $,  $ N\to \infty $; then an upper bound on $ \e_N^*(\tilde A) $ is produced by taking the union of configurations $ \omega_{N_m}^* \subset A_m $ for which $ \e_N(\omega_{N_m}^*) \leq \e_N^*(A_m) + 1 $. Indeed, by the short-range assumption \ref{it:shortrange}, and the same argument as above applied to $ \bigcup_m \omega_{N_m}^* $,
    \begin{equation}
        \label{eq:upper_Atilde}
        \begin{aligned}
        \limsup_{N\to \infty} \frac{\e_N^*(\tilde A)}{N^{1+s/d}}
        &\leq 
        \sum_{m} \left(\frac{\h_d(K_m)}{\sum \h_d(K_m)}\right)^{1+s/d} \cdot \limsup_{N\to \infty}\frac{\e_{N_m}^*(A_m) }{N^{1+s/d}} \\
        &\leq 
        \sum_{m} \left(\frac{\h_d(K_m)}{\sum \h_d(K_m)}\right)^{1+s/d} \cdot (1+\epsilon)\limsup_{N\to \infty}\frac{\e_{N_m}^*(K_m) }{N^{1+s/d}} \\
        &=(1+\epsilon) \cdot \frac{C_\e}{\left(\sum_m\h_d(K_m)\right)^{s/d}} \leq (1+\epsilon)^{1+s} \cdot \frac{C_\e}{\left(\h_d(\tilde A)\right)^{s/d}}.
        \end{aligned}
    \end{equation}
    Here we used \ref{it:Rdlimit} and that $ \{ \psi_m \} $ are bi-Lipschitz with the constant $ (1+\gamma) \leq (1+\epsilon) $.
    Finally, the substitution of \eqref{eq:lower_Atilde}--\eqref{eq:upper_Atilde} into \eqref{eq:upper_A}--\eqref{eq:lower_A} yields 
    \[
        (1-\epsilon) (1+\epsilon)^{-1} \frac{C_\e}{(\h_d(A))^{s/d}}
        \leq \liminf_{N\to \infty} \frac{\e_N^*(A)}{N^{1+s/d}} \leq \limsup_{N\to \infty} \frac{\e_N^*(A)}{N^{1+s/d}}\leq (1+\epsilon)^{1+s} \frac{C_\e}{(\h_d(A) - \epsilon)^{s/d}},
    \]
    so that  taking $ \epsilon \downarrow 0 $ finishes the proof of the lemma.
\end{proof}
We are now in the position to prove our main theorem for the case with no weight or external field. Note, we have verified that $ E_s^k $ satisfies the properties formulated in Section~\ref{sec:outline}, and thus Lemma~\ref{lem:federer} applies.

\begin{theorem}
    \label{thm:poppy_seed_k}
    Suppose $ A \subset \mathbb R^p $ is a compact $ (\mathcal H_d,d) $-rectifiable set with $ \m_d(A) = \h_d(A) $, $ s>0 $, $ k\geq 1 $. Then
    \begin{equation}
        \label{eq:poppy_seed}
        \lim_{N\to\infty}  \frac{ \mathcal{E}^k_s(A,N)}{N^{1+s/d}} = \frac{C^k_{s,d}}{\mathcal H_d(A)^{s/d}},
    \end{equation}
    where the constant $ C_{s,d}^k $ was introduced in~\eqref{CsdkDef}.
\end{theorem}
This is a special case of asymptotics from Theorem~\ref{thm:k_asympt}, in which $ \rho(x) \equiv 1/\mathcal H_d(A) $ and \(V\equiv 0\). We obtain the result about limiting distribution in the general situation (with weight and external field) below, in Section~\ref{sec:wt_fld}.
\begin{proof}
    This proof uses the approach developed in \cite{paper2} for general short-range interactions.  We proceed by establishing the asymptotics for unions of cubes, then for compact sets in $ \mathbb R^d $ ($p=d$ case), and finally proving that \eqref{eq:poppy_seed} holds for compact $ (\mathcal H_d,d) $-rectifiable sets $A$ via Lemma~\ref{lem:federer}.
    
    Consider the case $ A = \bigcup_1^M Q_m $, a union of equal closed disjoint cubes. Fix a sequence $ \omega_N^* $, for which $ E_s^k(\omega_N^*) < \mathcal E_s^k(A,N)+1$. Passing to a subsequence if necessary, it can be assumed that the following limits exist
    \[
        \beta_m := \lim_{N\to \infty} \frac{N_m }{N}, \qquad 1\leq m \leq M,
    \]
    where we set $ N_m = \# (\omega_N^*\cap Q_m) $.
    On the one hand,
    \[
        \begin{aligned}
            \liminf_{N\to \infty} \frac{\mathcal E_s^k(A,N)+1}{N^{1+s/d}} 
        \geq& 
        \liminf_{N\to \infty} \frac{E_s^k(\omega_N^*)}{N^{1+s/d}}
        =\liminf_{N\to \infty} 
        \frac{E_s^k(\omega_N^*)}{ \sum_{m} E_s^k(\omega_N^*\cap Q_m) } \cdot 
        \frac{ \sum_{m} E_s^k(\omega_N^*\cap Q_m) }{N^{1+s/d}}
        \\
        &\geq \sum_{m} \liminf_{N\to \infty}\frac{\mathcal  E_s^k(Q_m, N_m) }{N^{1+s/d}}
        = \sum_{m} \beta_m^{1+s/d} \liminf_{N\to \infty}\frac{\mathcal E_s^k(Q_1, N_m) }{N_m^{1+s/d}}
        \\
        &\geq \liminf_{N\to \infty}\frac{\mathcal E_s^k(Q_1, N) }{N^{1+s/d}} \cdot \sum_{m} \beta_m^{1+s/d} = \frac{C_{s,d}^k}{\h_d(Q_1)} \cdot \sum_{m} \beta_m^{1+s/d},
    \end{aligned}
    \]
    where we used the short-range property for $ E_s^k $, Corollary~\ref{cor:cube_general}, and that all the cubes $ Q_m $ are equal, so the value of $ \mathcal E_s^k(Q_m, N) $ is independent of $ m $. The minimum of $ \sum_m \beta_m^{1+s/d} $  over nonnegative $ \beta_m $ with $ \sum_m \beta_m =1  $ is obtained for $ \beta_1=\ldots=\beta_M = 1/M $.

    On the other hand, given $ \beta_m = 1/M $,  $ 1\leq m \leq M $, it suffices to place in $ Q_m $ a configuration $ \omega_{N_m}^* $ of cardinality $ N_m = \lceil N/M \rceil $ with $ E_s^k(\omega_{N_m}^*) < \mathcal E_s^k(Q_m,N_m) + 1 $ to obtain
    \[
        \begin{aligned}
            \limsup_{N\to \infty} \frac{\mathcal E_s^k(A,N)}{N^{1+s/d}} 
        \leq& 
        \limsup_{N\to \infty} 
        \frac{ E_s^k\left(\bigcup_{m} \omega_{N_m}^*\right) }{N^{1+s/d}}
        =\limsup_{N\to \infty} 
        \frac{E_s^k\left(\bigcup_{m} \omega_{N_m}^*\right)}{ \sum_{m} E_s^k(\omega_{N_m}^*) } \cdot 
        \frac{ \sum_{m} E_s^k(\omega_{N_m}^*) }{N^{1+s/d}}
        \\
        &\leq \sum_{m} \limsup_{N\to \infty}\frac{E_s^k(\omega_{N_m}^*) }{N^{1+s/d}}
        = \sum_{m} \beta_m^{1+s/d} \limsup_{N\to \infty}\frac{E_s^k(\omega_{N_m}^*) }{N_m^{1+s/d}}
        \\
        &= \limsup_{N\to \infty}\frac{\mathcal E_s^k(Q_1, N) }{N^{1+s/d}}\cdot \sum_{m} \beta_m^{1+s/d}= \frac{C_{s,d}^k}{\h_d(Q_1)} \cdot M^{-s/d}.
    \end{aligned}
    \]
    To summarize, the last two displays prove that asymptotics for the union of $ M $ equal disjoint cubes is $ M^{-s/d} $ times the asymptotics for one such cube, in agreement with~\eqref{eq:poppy_seed}.

    The case of a union of closed equal cubes with disjoint interiors (but not necessarily disjoint themselves) follows as an application of stability and monotonicity, by approximating the cubes of the union from the inside with concentric disjoint equal cubes. We shall omit the details and instead discuss obtaining~\eqref{eq:poppy_seed} for general compact sets from unions of cubes with disjoint interiors. The omitted argument follows the same lines.  

    It suffices to assume $ \h_d(A) > 0 $, since otherwise $ A $ can be covered with a union of equal cubes of arbitrarily small measure, and monotonicity property directly implies that the limit from~\eqref{eq:poppy_seed} is infinite. Now fix an $ \epsilon >0 $.
    For $ \delta=\delta(\epsilon, s,k,p,d) $ as in Lemma~\ref{lem:stable} (note that $ \delta$ is set-independent due to \(p=d\)!), let $ J_\epsilon \supset A $ be a finite union of closed equal dyadic cubes with disjoint interiors, such that $  \h_d(A) > (1 - \delta )\h_d(J_\epsilon) $; then formula \eqref{eq:poppy_seed} applies to $ J_\epsilon $. Without loss of generality, $ \delta \leq \epsilon < 1 $. On the one hand, monotonicity property together with asymptotics on $ J_\epsilon $ give
    \[
        \liminf_{N\to \infty} \frac{\mathcal E_s^k(A,N)}{N^{1+s/d}} \geq \lim_{N\to \infty} \frac{\mathcal E_s^k(J_\epsilon,N)}{N^{1+s/d}} =\frac{C_{s,d}^k}{\h_d(J_\epsilon)^{s/d}} \geq (1-\epsilon)^{{s/d}} \frac{C_{s,d}^k}{\h_d(A)^{s/d}}.
    \]
    In addition, by the choice of $ J_\epsilon $ and the stability property from Lemma~\ref{lem:stable} for the pair of sets $ A\subset J_\epsilon $,
    \[
        (1-\epsilon)\limsup_{N\to \infty} \frac{\mathcal E_s^k(A,N)}{N^{1+s/d}} \leq \lim_{N\to \infty} \frac{\mathcal E_s^k(J_\epsilon,N)}{N^{1+s/d}} =\frac{C_{s,d}^k}{\h_d(J_\epsilon)^{s/d}} \leq\frac{C_{s,d}^k}{\h_d(A)^{s/d}},
    \]
    which completes the proof when $ A $ is a general compact subset of $ \mathbb R^d $.

    To obtain the desired result for $ (\h_d,d) $-rectifiable $ A \subset \mathbb R^p $ with $ \h_d(A) = \m_d(A)  > 0 $, recall Lemmas~\ref{lem:short_range} and~\ref{lem:stable} and apply Lemma~\ref{lem:federer} to the sequence of functionals $ E_s^k $ on \(N\)-point configurations.
    It thus remains to discuss the case $ \h_d(A) = \m_d(A) =0 $. We can argue by contradiction: suppose 
    \[
        \liminf_{N\to\infty}  \frac{ \mathcal{E}^k_s(A,N)}{N^{1+s/d}} = C < \infty,
    \]
    and let $ \{ n : {n \in \mathscr N} \} $ be the subsequence along which the $ \liminf $ is attained. Without loss of generality, $ k=1 $. Let further $ \omega_{n}^* $ be the sequence of configurations with $ E_s^1(\omega_{n}^*) < \mathcal E_s^1(A,n)+1$ for every $ n \in \mathscr N $. It follows that for $ n_0 $ sufficiently large,
    \[
        \sum_{x\in \omega_n} \|x-(x;\omega_{n}^*)_1\|^{-s} \leq 2C n^{1+s/d}, \qquad n \geq n_0,
    \]
    implying that for any $ \gamma > 0 $, the number of elements $ x\in \omega_{n}^* $ such that $ \|x-(x;\omega_{n}^*)_1\| < \gamma n^{-1/d} $  is at most $ 2C \gamma^s n $. Taking $ \gamma $ small enough gives at least $ n(1-2C\gamma^s) $ elements $ x\in\omega_n $ for which $  \|x-(x;\omega_{n}^*)_1\| \geq \gamma n^{-1/d} $. Denote the subconfiguration of such $ x $ by $ \omega_n' \subset \omega_n $. It follows that the balls $ \{ B(x,\, 2^{-1}\gamma n^{-1/d}) : x\in\omega_n' \} $ have disjoint interiors, and so in view of 
    \[
        \bigcup_{x\in \omega_n'} B\left(x,\, 2^{-1}\gamma n^{-1/d}\right) \subset A_{\gamma n^{-1/d}},
    \]
    for any $ \epsilon > 0 $ and all large enough $ n_1=n_1(\epsilon) $, there holds
    \[
        \# \omega_n' v_p(2^{-1}\gamma n^{-1/d})^p \leq  (\m_d(A) + \epsilon) v_{p-d} (\gamma n^{-1/d})^{p-d}, \qquad n \geq n_1.
    \]
    Using that $ \# \omega_n' \geq (1-2C\gamma^s)n $ and $ \m_d(A) = 0 $, from the last display we have 
    \[
        (1-2C\gamma^s) v_p 2^{-p}\gamma^d \leq v_{p-d} \epsilon,
    \]
    which is the desired contradiction when $ \epsilon > 0 $ is sufficiently small. This proves \(\liminf_{N\to\infty}  \frac{ \mathcal{E}^k_s(A,N)}{N^{1+s/d}} = \infty\) for the case $ \h_d(A) = \m_d(A) =0 $.
\end{proof}

\subsection{Adding a multiplicative weight and external field}
\label{sec:wt_fld}
We will now extend the results of the previous section by introducing an external field and a weight, so that the problem at hand is optimization of the functional~\eqref{Edef}. An essential ingredient in the proof is the partitioning of the set $ A $ according to the values of $ V $; similar ideas were used by the authors in \cite{hardinGenerating2017}.  
First, some remarks about the positivity of weight $ w $ and external field $ V $ are in order.
\begin{remark}
    \label{rem:wVpositive}
    Since $ V $ is assumed to be lower semicontinuous and $A$ is compact, it follows that $V$ is bounded below on $ A $ and, by adding a suitable constant, we may assume $ V\geq 0 $.
    Furthermore, by the definition of a CPD-weight $ w $, there exist positive numbers $ \delta $ and $w_0$ such that for any pair $ x,y\in A $ satisfying $ \|x-y\| \leq \delta $ we have $ w(x,y) \geq w_0 $. Let $ \{ B_j \}_1^n $ be a covering of $ A $ with $n=n(\delta)$ balls  of radius $ \delta/2 $.  For $ \omega_N \in A^N $ and any ball $ B_j $ containing at least $ k+1 $ elements from $ \omega_N $, we have $ \|x-y\| \leq \delta $ for  $x\in \omega_N\cap B_j$ and $y\in \nn_k(x,\omega_N)$. On the other hand, there are at most $ nk^2 $   pairs $ x,y $ with $x\in \omega_N\cap B_j$ and $y\in \nn_k(x,\omega_N)$ satisfying $ \|x-y\| > \delta $  since  in such a case  $ x $ must belong to a ball $ B_j $ with at most $ k $ elements.     
    Defining $\widetilde w:=\max \{w,w_0\}$, we then have
$$  \lim_{N\to \infty} \frac{E_s^k(\omega_N,\widetilde w)-E_s^k(\omega_N,  w)}{N^{1+s/d}} =0,$$
 since 
 $$0\le E_s^k(\omega_N,\widetilde w)-E_s^k(\omega_N,  w)\le \sum_{x\in\omega_N}\sum_{\substack{y\in\nn_k(x;\omega_N),\\ \|x-y\| > \delta}} w_0\|x-y\|^{-s}\le nk^2w_0\delta^{-s}.$$
 Hence, for the purpose of asymptotics, it may be assumed  $ w\geq w_0 $. We employ this fact in the following useful proposition. 
 \end{remark}
 
 \begin{proposition}
    \label{prop:abs_cont}
    Let the assumptions of Theorem~\ref{thm:poppy_seed_k} hold and \(w\) be a CPD-weight.
     If $ \omega_N \in A^N $, $ N\geq 1 $, is a sequence of configurations for which
    \[
        \limsup_{N\to \infty} \frac{E_s^k(\omega_N; w)}{N^{1+s/d}} < +\infty,
    \]
    then any cluster point of $ \nu(\omega_N) $ is absolutely continuous with respect to $ \h_d $.
\end{proposition}
\begin{proof}
As remarked above, we may assume $w \ge w_0>0$.   
  Then  $E_s^k(\omega_N; w) \geq w_0 E_s^k(\omega_N) $, and we may argue as in \cite[Lemma 4.9]{hardinGenerating2017}.
\end{proof}
\begin{remark}
    \label{rem:offdiagonal}
    The discussion in Remark~\ref{rem:wVpositive} can also be used to estimate the terms in $ E_s^k $ for pairs $ (x,y) $ with a fixed positive separation. Namely, by property (c) of CPD-weight, for every $ \delta > 0 $ there exists a constant $ M_\delta $ such that $ \|x-y\| \geq \delta $ implies $ 0\leq w(x,y) \leq M_\delta $. Using a covering of $ A $ with $ n=n(\delta) $ balls of radius $ \delta/2 $ in the same way as in Remark~\ref{rem:wVpositive}, we conclude 
    \begin{equation*}
        0 \leq \sum_{x\in\omega_N}\sum_{\substack{y\in\nn_k(x;\omega_N),\\ \|x-y\| > \delta}} w(x,y)\|x-y\|^{-s} \leq \sum_{x\in\omega_N}\sum_{\substack{y\in\nn_k(x;\omega_N),\\ \|x-y\| > \delta}} M_\delta\|x-y\|^{-s}\le nk^2M_\delta\delta^{-s}.
    \end{equation*}
\end{remark}
Using the above observations we further derive an analog of the short-range property~\eqref{eq:short_range} for weighted interactions. For the purposes of the proof of the main theorem, it will be enough to establish an inequality corresponding to the local behavior of asymptotically optimal configurations. For the rest of the section, let
\begin{equation}
    \label{eq:U_defined}
    \U(\omega_N; S) := \sum_{\substack{(x,y)\in \Lambda_k(\omega_N),\\ x\in S}} \left(w(x,y)\|x-y\|^{-s} +N^{s/d} V(x)\right)
\end{equation}
be the sum of terms in  $ E_s^k(\omega_N; w,V) $ corresponding to edges of $ \Lambda_k(\omega_N) $, originating from the entries $ x\in \omega_N\cap S $ with $ S\subset A $. Notice that as a function of  set $ S $, $\U(\omega_N; S)$ is a positive measure. Thus, for any sequence of configurations $ \omega_N $, $N\geq 1,$ with $ \limsup_{N} \U(\omega_N; A) /N^{1+s/d} < \infty $, up to passing to a subsequence there exists a weak$ ^* $ limit of measures $ \U(\omega_N; \cdot\,) /N^{1+s/d} $.
\begin{lemma}
    \label{lem:wt_short_range}
    Let the assumptions of Theorem~\ref{thm:k_asympt} be satisfied, with \(w(x,x)+V(x)\) finite on a subset of \(A\) of positive \(\mathcal H_d\)-measure.

    Suppose that $ \{\omega_N^*\}_1^\infty $ is a sequence of $(k,s,w,V)$-asymptotically optimal configurations on $ A $, with
    \[
        \begin{aligned}
            \frac{\U(\omega_N^*;\ \cdot)}{N^{1+s/d}} &\weakto \lambda\\ 
            \nu(\omega_N^*) &\weakto \mu,
        \end{aligned}
    \]
    for a finite measure $ \lambda $ and a probability measure $ \mu $.
    Let $ x_1, x_2 \in \supp \mu $ and $ B_m:= B(x_m, r_m) $, $ m=1,2 $, be disjoint and such that $ \h_d(\partial B_m \cap A) =0 $, $ m=1,2 $. In addition, let $ w(x,y) $ be bounded when $ x,y \in B(x_m, 3r_m) $.

    Then for any compact $ S_m\subset A\cap B_m $, $ m=1,2 $, and $ B = B_1\cup B_2 $, 
    \begin{equation}
        \label{eq:lem_wt_short}
        \limsup_{N\to \infty} \frac{\U(\omega_N^*; B)}{N^{1+s/d}} 
        \leq \min_{\alpha_1 + \alpha_2 =\mu(B) } \sum_{m=1,2}\left( \alpha_m^{1+s/d} \c^k \frac{\sup_{x,y\in S_m} w(x,y)}{\h_d(S_m)^{s/d}} + \alpha_m\sup_{x\in S_m} V(x)\right).
   \end{equation}
\end{lemma}
\begin{proof}
    Notice that, due to \(w(x,x)+V(x)\) being finite on a subset of positive measure, 
    $$\limsup_{N\to \infty}  {\mathcal E_s^k(A,N;w,V)}/{N^{1+s/d}} < \infty,$$
    so by Proposition~\ref{prop:abs_cont} we have $\mu \ll \h_d$. Subsequently, $ \mu(\partial B_m \cap A) =0 $, \(m=1,2\).
    We shall present a sequence of configurations $ \omega_N' $ for which the right-hand side in~\eqref{eq:lem_wt_short} corresponds to \(\U(\omega_N'; B)\). The inequality will then follow by the asymptotic optimality of the sequence \( \omega_N^*\). 

    Fix $ \gamma \in (0,1) $ and $\alpha_1, \alpha_2\ge 0$ such that $\alpha_1 + \alpha_2 =\mu(B)$.  
    We further denote $ \gamma B_m := B(x_m, \gamma r_m) $ and $\gamma B := \gamma B_1 \cup \gamma B_2 $.  
    Let $n=\#(\omega_N^*\cap B)$, $ n_1 := \min (\lfloor\alpha_1 N\rfloor, n) $, and $ n_2:= n - n_1 $ and note that  $n_m/N\to \alpha_m$ as $N\to\infty$ for $m=1,2$, since $n/N\to \mu(B)$ by weak* convergence and the equality $ \mu(\partial B \cap A) =0 $.  
    Choosing an $n_m$-point configuration $ \omega_{n_m} $ in $ S_m\cap \gamma B_m $, such that $ E_s^k(\omega_{n_m}; w,V) \leq \mathcal E_s^k(S_m \cap \gamma B_m, n_m; w,V) + 1 $ for $m=1,2$, let $\omega_n=\omega_{n_1}\cup
    \omega_{n_2}$ and define $  \omega_N' $ as
    \begin{equation}
        \label{eq:tilde_omega_def}
         \omega_N' := \omega_n \cup (\omega_N^*\setminus B).
    \end{equation}
    By \eqref{eq:U_defined},
    \begin{equation}
        \label{eq:decomposition}
        \begin{aligned}
            E_s^k(\omega_N^*; w, V) &= \U(\omega_N^*, B) + \U(\omega_N^*, A\setminus B),\\
            E_s^k(\omega_N'; w, V)  &= \U(\omega_N', B) + \U(\omega_N', A\setminus B).
    \end{aligned}
    \end{equation}
    To obtain inequality~\eqref{eq:lem_wt_short}, we need to establish the reverse inequality between
    the asymptotics of $ \U(\omega_N^*, A\setminus B) $ and that of $ \U(\omega_N', A\setminus B) $. By~\eqref{eq:tilde_omega_def}, these sums differ only by the terms corresponding to pairs $ (x,y) \in \Lambda_k(\omega_N^*)$ with $ x\in \omega_N^*\setminus B $ and $ y\in \omega_N^*\cap B $, which  are replaced in $\Lambda_k(\omega_N')$ by pairs $ (x,y') $, where $ y'\neq y $.  
    That is, we have
    \begin{equation}
        \label{eq:decompose_diff}
        \U(\omega_N^*, A\setminus B) - \U( \omega_N', A\setminus B) =
        \sum_{\substack{x\in \omega_N^*\setminus B,\ y\in B, \\ (x,y)\in \Lambda_k(\omega_N^*),\  (x,y')\in \Lambda_k(\omega_N') }} 
        \left(
            \frac{w(x , y ) }{\| x - y \|^{s}} - \frac{w(x , y' ) }{\| x - y' \|^{s}} 
        \right).
    \end{equation}
    Here \(y'\) is the new nearest neighbor, acquired by \(x\) instead of \(y\). Our goal is to estimate the asymptotics of this difference from below. Denote the sets of new and old edges $ (x,y) $ and $ (x,y') $ in the right-hand side by $ \Lambda_N \subset\Lambda_k(\omega_N) $ and $ \Lambda_N' \subset\Lambda_k(\omega_N') $, respectively.
    Note that $ \|x-y'\| \geq \|x-y\| $ holds for every pair of corresponding edges $ (x,y) $ and $ (x,y') $ (unless $ y'\in \gamma B $, in which case distances $\|x-y'\|$ are uniformly bounded below, and do not contribute to the asymptotics).

    For every fixed $ r> 0 $ we can further decompose the sum in the right-hand side of \eqref{eq:decompose_diff} depending on whether $ \| x- y'\| > r  $ :
    \[
       \Sigma_1 + \Sigma_2 := \left(
            \sum_{\substack{(x,y)\in \Lambda_N,\  (x,y')\in \Lambda_N'\\ \|x-y'\| > r }}
            + 
            \sum_{\substack{(x,y)\in \Lambda_N,\  (x,y')\in \Lambda_N'\\ \|x-y'\| \leq r }}
        \right) 
        \left(
            \frac{w(x , y ) }{\| x - y \|^{s}} - \frac{w(x , y' ) }{\| x - y' \|^{s}} 
        \right).
    \]
    In the sum \(\Sigma_1\), points \(x\) and \(y'\) are positively separated, so by Remark~\ref{rem:offdiagonal} its negative terms do not contribute to the asymptotics, and $ \liminf_{N\to \infty}\Sigma_1/N^{1+s/d} \geq 0 $.  

    To estimate the second sum, let $ r < \min(r_1, r_2) $. Since $ y\in B $ and inequalities $ r\geq  \| x - y'\| \geq \| x - y\| $ hold for pairs of corresponding edges $(x,y) \in \Lambda_N $ and $(x, y')\in \Lambda_N'$, for each term of \(\Sigma_2\) we have $ y' \in B(x_m, 3r_m) $ whenever $ y \in B(x_m, r_m) $, $ m=1,2 $.
    Because $ w(x,y) $ was assumed bounded for all $ x,y \in B(x_m, 3r_m) $ by some constant $ M_w > 0 $,  it follows that for the $ w_0 $ from Remark~\ref{rem:wVpositive},
    \begin{equation}
        \label{eq:compare_omega_prime}
        \frac{w(x , y' ) }{\| x - y' \|^{s}}
        \bigg/
        \frac{w(x , y ) }{\| x - y \|^{s}}
        \leq \frac{M_w}{w_0}\qquad  \text{when } (x,y)\in \Lambda_N,\  (x,y')\in \Lambda_N', \text{ and } \|x-y'\|\leq r.
    \end{equation}
    Finally, we can estimate the negative terms $ - {w(x , y' ) }/{\| x - y' \|^{s}} $ in $ \Sigma_2 $. By~\eqref{eq:compare_omega_prime}, each is at most a constant multiple of the corresponding positive term $ {w(x , y ) }/{\| x - y \|^{s}} $. In the latter, $ r \geq \|x-y\|  $, $ y\in B $, and so every starting point $ x \in (B_r\setminus B) $.
    Using the definition of measure $ \lambda $ and \eqref{eq:compare_omega_prime},  we conclude for $ N $ large enough,
    \[
        \frac{\Sigma_2}{N^{1+s/d}} \geq  -\frac{M_w}{w_0}  
        \sum_{\substack{(x,y)\in \Lambda_k(\omega_N^*)\\ x\in (B_r\setminus B) }} \frac{w(x , y ) }{\| x - y \|^{s}} \geq -\frac{2M_w}{w_0}\lambda(B_r\setminus B).
    \]
    Observe that because $ \lambda $ is finite, $ \lim_{r\downarrow0} \lambda(B_r\setminus B) = \lambda(B\setminus B) = 0 $, and the right-hand side can be made smaller than any given $ \epsilon > 0 $. It follows that in \eqref{eq:decompose_diff},
    \[
        \liminf_{N\to \infty} \frac{\U(\omega_N^*, A\setminus B) - \U( \omega_N', A\setminus B)}{N^{1+s/d}} \geq 0.
    \]
    Combined with equation~\eqref{eq:decomposition}, asymptotic optimality of $ \omega_N^* $ now implies
    \[ 
        \limsup_{N\to \infty}\frac{\U(\omega_N^*, B)}{N^{1+s/d}}\leq \liminf_{N\to \infty}\frac{\U(\omega_N', B)}{N^{1+s/d}}.
    \]
    Writing $ S^\gamma_m :=S_m\cap \gamma B_m $, by the separation between $ \gamma B $ and sets $ A\setminus B $, and $ B_m $ being disjoint, we infer from the above inequality and Theorem~\ref{thm:poppy_seed_k}: 
    \[ 
        \begin{aligned}
            \limsup_{N\to \infty} \frac{\U(\omega_N^*; B)}{N^{1+s/d}} 
            &\leq  \liminf_{N\to \infty} \sum_{m=1,2} \left(\frac {n_m} N\right)^{1+s/d} \frac{E_s^k(\omega_{n_m}; w,V)}{n_m^{1+s/d}} \\
            &= \sum_{m=1,2} \alpha_m^{1+s/d} \, \frac{\sup_{x,y\in S^\gamma_m} w(x,y)}{\h_d(S^\gamma_m)^{s/d}} + \alpha_m \sup_{x\in S_m^\gamma} V(x).
    \end{aligned}
    \]
    Taking $ \gamma\uparrow 1 $ gives~\eqref{eq:lem_wt_short}, since $ \h_d(\partial B \cap A) =0 $.
\end{proof} 
Before proving Theorem~\ref{thm:k_asympt}, let us discuss an alternative form of condition (a) in the definition of a CPD weight. It transpires from the proofs of Lemma~\ref{lem:wt_short_range} and Theorem~\ref{thm:k_asympt} that in place of $ \h_d $-a.e.\ continuity in condition (a) one can assume
\begin{itemize}
    \item[(a$ ' $)]
        $w$ is bounded on $ A\times A $, lower semi-continuous (as a function on $A\times A$) at $ \h_d $-a.e.\ point of the diagonal $\diag(A):=\{(x,x) : x\in A\}$, and such that for $ \h_d $-a.e. $ x\in A $ and any $ \epsilon > 0 $ there
        is an $ r_x' > 0 $ such that for every $ r $, $0< r < r_x' $, there exists a closed set $ A_{x,r}\subset A\cap B(x,r) $ for which $ \h_d(A_{x,r}) \geq (1-\epsilon) \h_d [A\cap B(x,r)] $ and
        \begin{equation*}
            w(y,z) \leq w(x,x)+\epsilon, \qquad y,z \in A_{x,r}.
        \end{equation*} 
\end{itemize} 
In particular, condition (a$'$) holds if $ w $ is symmetric and lower semicontinuous on $ A\times A $ and $ \h_d $-a.e.\ point of $ \diag (A) $ is a Lebesgue point for $ w(x,y) $ with respect to the measure $ \h_d\otimes \h_d $ on $A\times A$.   Another version of (a), not requiring boundedness on the diagonal, is as follows.
\begin{itemize}
    \item[(a$ '' $)]
        $w$ is a marginally radial weight, lower semi-continuous at $ \h_d $-a.e.\ $ x\in \diag(A)$, and such that for $ \h_d $-a.e. $ x\in A $ and any $ \epsilon > 0 $ there is an $ r_x' > 0 $ and a closed $ A_{x,r} \subset A\cap B(x,r) $, $ 0 < r < r_x' $, as in property (a$ ' $).
\end{itemize} 

\begin{proof}[Proof of Theorem~\ref{thm:k_asympt}]
    Let the assumptions of Theorem~\ref{thm:k_asympt} hold. In view of Remark~\ref{rem:wVpositive}, we hereafter let $V\geq 0$ and suppose there is a $w_0>0$ such that $w\geq w_0$ on $A\times A$.

    If $ \h_d(A) = 0 $,  it follows from the latter assumption  and Theorem~\ref{thm:poppy_seed_k} that  
    $$ \lim_{N\to \infty}  \frac{\mathcal E_s^k(A,N;w,V)}{N^{1+s/d}} = +\infty, $$ and so there is nothing to prove.
    Now let $ \h_d(A) > 0 $ and $ w(x,x) + V(x) < \infty $ on a closed subset of $A$ of positive $ \h_d $-measure.  Minimizing $ E_s^k $ on this subset gives an upper bound on $ \mathcal E_s^k(A,N;w,V) $, implying that 
    \begin{equation}
        \label{eq:asymp_is_finite}
        \limsup_{N\to \infty}  \frac{\mathcal E_s^k(A,N;w,V)}{N^{1+s/d}} < \infty.
    \end{equation}
    Fix $ 0<\epsilon <w_0/3 $. 
    For $r>0$, denote by $ B_A(x,r) : = A\cap B(x,r) $ the ball of radius $ r $ relative to $ A $, centered at a point $ x $.

    By the (semi)continuity properties of $ V $ and $ w $, for almost every $ x\in A $ and a sufficiently small $ r_x^{(1)} $,  if $ y,z \in B_A(x,r_x^{(1)}) $ then (when e.g.\ $ w(x,x) = + \infty $, apply the appropriate modifications) 
    \begin{equation}
        \label{eq:lower_semi}
        |w(y,z) - w(x,x)| \leq \epsilon \qquad V(y) \geq V(x)-\epsilon.
    \end{equation}
    Furthermore, the set $ A $ can be partitioned according to the values of $ V $ into the subsets 
    \[
        \begin{aligned}
            A_{l} &:= \left\{x\in A: l\epsilon \leq V(x) < (l+1)\epsilon \right\}, \qquad 0\leq l \leq M-1,\\
            A_{M} &:= \left\{x\in A: M\epsilon \leq V(x)\right\},
    \end{aligned}
    \]
    with $ M $ chosen so that $ \h_d(A_{M}) < \epsilon $. Thus $ A = \bigsqcup_0^M A_l $.

    Applying the Lebesgue density theorem \cite[Corollary 2.14]{mattila1995geometry} to each $ A_{l} $ gives that for  $ \h_d $-almost every $ x\in A_{l} $ there exists some $ r_x^{(2)} > 0 $ such that for every $ r < r_x^{(2)} $,
    \[
        \h_d[A_{l} \cap B(x,r)] \geq (1-\epsilon) \h_d[B_A(x,r)],
    \]
    implying, since every $ x \in A $ is in exactly one $ A_l $, that for $ \h_d $-a.e.\ $ x \in A $ and $ r < r_x^{(2)} $:
    \begin{equation*}
        \h_d(\{ y \in B_A(x,r): V(y) \leq V(x) + \epsilon \}) \geq (1-\epsilon) \h_d[B_A(x,r)].
    \end{equation*}
    Thus for $ \h_d $-a.e.\ $ x \in A $ and $ r < \min( r_x^{(1)}, r_x^{(2)} ) $, there is a closed set $ A_{x,r} \subset B_A(x,r) $ satisfying $ \h_d[A_{x,r}] \geq (1-\epsilon) \h_d[B_A(x,r)] $ and
    \begin{equation}
        \label{eq:upper_semi}
        |w(y,z) - w(x,x)|\leq \epsilon \qquad V(y) \leq V(x)+\epsilon, \qquad y,z\in A_{x,r}.
    \end{equation}

    Let $ \omega_N^* \in A^N $, $N\geq 1$,  be a $(k,s,w,V)$-asymptotically optimal sequence and 
    let $ \mu $ and $ \lambda $ denote some cluster points of the sequences of measures $\nu (\omega_N^*)$ and $ \U(\omega_N^*; \ \cdot)/N^{1+s/d} $, respectively. The latter exists by \eqref{eq:asymp_is_finite}. Also by \eqref{eq:asymp_is_finite}, assumption $ w\geq w_0 $, and Proposition~\ref{prop:abs_cont}, it follows that $ \mu \ll \h_d $. In addition, 
    both $ \mu $ and $ \mathcal H_d $ are Radon measures since $ A $ is a complete metric space \cite[Theorem 7.1.7]{bogachevMeasure2007}.
    The differentiation theorem for Radon measures \cite[Theorem 2.12]{mattila1995geometry} implies that for $ \h_d $-a.e. $ x \in A $ there exists an $ r_x^{(3)} > 0 $ such that whenever $ r < r_x^{(3)} $, we have
    \begin{equation}
        \label{eq:centers}
        \begin{aligned}
        \left|\frac{\mu[B(x,r)]}{\h_d[B_A(x,r)]} - \frac{d\mu}{d\h_d}(x)\right| < \epsilon, \qquad
        & \text{and}\\
        1- \epsilon < \frac{\mu[B(x,r)]}{\h_d[B_A(x,r)]} \bigg/ \frac{d\mu}{d\h_d}(x) < 1+ \epsilon, \qquad 
        &\frac{d\mu}{d\h_d}(x) > 0.
        \end{aligned}
    \end{equation}

    Setting for $ \h_d $-a.e. $ x\in A $  the quantity $ r_x := \min \{ r_x^{(1)}, r_x^{(2)}, r_x^{(3)}\}/2 $, it follows that the properties  \eqref{eq:lower_semi}--\eqref{eq:centers} hold for $ \h_d $-a.e.\ $ x\in A $ and  closed balls $ B(x,r) $ of radius $ r< 2r_x $. Denote the set of such $ x $  by $ \tilde A $.

    In the next part of the proof we shall derive two-sided estimates for the asymptotics of $ \U(\omega_N^*; \ \cdot ) $ on sequences of balls $ B_j^{(m)} $, $ m=1,2 $, $ j\geq 1$, shrinking to a pair of fixed points $ x_1, x_2 \in \tilde A $. This will allow to derive estimates for the densities $ d\mu/d\h_d(x_m) $, $ m=1,2 $.
    Fix a pair of elements $ x_1\neq x_2 \in \tilde A \cap \supp \mu $. We will consider two sequences of balls relative to $ A $: 
    $$B_j^{(m)} = B_A(x_m,r_j^{(m)}),\qquad  m=1,2,\ j\geq 1, $$
    with vanishing radii $ r_j^{(m)}\downarrow0 $. Without loss of generality, $ r_j^{(m)} \leq \min(r_{x_1}, r_{x_2}) $ and $ B_j^{(m)} $ are positive distance apart.
    Because $ \h_d(A) <\infty $, the sequences of radii $ r_j^{(m)} $ can further be chosen to satisfy $ \h_d(\partial B_j^{(m)} \cap A) = 0$, $ m=1,2 $, since finiteness of $ \h_d(A) $ implies, at most a countable number of possible $ r_j^{(m)} $ have a positive value of $ \h_d(\partial{B_j \cap A}) $. Likewise, we chose $ r_j^{(m)} $ so that $ \lambda(\partial B_j^{(m)} \cap A) = 0 $.

    The absolute continuity of $\mu$ with respect to $\h_d$ then implies $ \mu(\partial B_j^{(m)} \cap A) = 0$, $ j\geq 1 $, $ m=1,2 $. By the weak-star convergence of $\nu (\omega_N^* )$ to $\mu$, the limits below exist:
    \begin{equation}\label{weakBj}
        \lim_{N\to \infty} \nu (\omega_N^*)[B_j^{(m)}] =\mu(B_j^{(m)}) =:\beta_j^{(m)}, \qquad m=1,2, \ j\geq 1. 
    \end{equation}
    We shall further estimate the asymptotics of $ \U(\omega_N^* ; B_j)/N^{1+s/d} $, for the set $ B_j := B_j^{(1)}\cup B_j^{(2)} $.
    With  $ w_m := w(x_m, x_m) $, $ V_m := V(x_m) $,
    remark that  for $x,y\in 2B_j^{(m)}$ (with $2B_j^{(m)} $ being the concentric relative ball of double radius), we have
    $  w(x,y)\geq w_m-\epsilon $ and  $ V(x)\ge V_m -\epsilon$. 
    Observe that by Remark~\eqref{rem:offdiagonal}, for $ j $ fixed,
    $$
    \Sigma^{(m)}_j := \sum_{x\in \omega_N^*\cap B_j^{(m)} } \sum_{y\in \nn_k(x;\omega_N^*)\setminus 2B_j^{(m)} } w(x,y ) \|x-y \|^{-s} = o(N^{1+s/d}), \qquad m=1,2.
    $$ 
    By dividing the edges $ (x,y) $ in $ \U(\omega_N^*; B_j) $ according to whether $ y\in 2B_j^{(m)} $ and using the previous display, we obtain:
    \begin{equation*}
        \begin{aligned}
            &\liminf_{N\to \infty} \frac{\U(\omega_N^*; B_j)}{N^{1+s/d}} 
            \geq \liminf_{N\to \infty} \frac{\U(\omega_N^*;w_m-\epsilon,V_m-\epsilon;\ B_j) - \Sigma^{(m)}_j}{N^{1+s/d}}
            \\
            &\quad \geq  \sum_{m=1,2} \liminf_{N\to \infty}                 \frac{E_s^k(\omega_N^* \cap {B_j^{(m)} }; w_m-\epsilon,V_m-\epsilon)}{N^{1+s/d}}, 
        \end{aligned}
    \end{equation*}
    where the first inequality estimates the weight and external field in $ \U $ by constants from below, using \eqref{eq:lower_semi}--\eqref{eq:centers}, and the second inequality is due to the distance to nearest neighbors not decreasing when passing from a configuration to its subconfiguration.
    Using Theorem~\ref{thm:poppy_seed_k} and \eqref{weakBj}, we  deduce
    \begin{equation}
        \label{eq:lower}   
        \begin{aligned}       
            \liminf_{N\to \infty}  \frac{\U(\omega_N^*; B_j)}{N^{1+s/d}}         
            & \geq \sum_{m=1,2} \liminf_{N\to \infty} \frac{E_s^k(\omega_N^* \cap {B_j^{(m)}}; w_m-\epsilon,V_m-\epsilon)}{N^{1+s/d}}\\
            & \geq \sum_{m=1,2} \liminf_{N\to \infty} \frac{E_s^k(\omega_N^* \cap {B_j^{(m)}}; w_m-\epsilon,V_m-\epsilon)}{\#(\omega_N^* \cap {B_j^{(m)})^{1+s/d}}}\left(\frac{\#(\omega_N^* \cap {B_j^{(m)}})}{N}\right)^{1+s/d}\\
            & \geq \sum_{m=1,2} 
            \left(
                \beta^{(m)}_j ( w_m-\epsilon )\cdot C_{s,d}^k\left(\frac{{ \beta^{(m)}_j} }{\h_d(B_j^{(m)})}\right)^{s/d} + \beta^{(m)}_j(V_m-\epsilon) 
            \right)\\
            & \geq \sum_{m=1,2} 
            \left(
                \beta^{(m)}_j ( w_m-\epsilon )\cdot C_{s,d}^k \left(\frac{d\mu}{d\h_d}(x_m) - \epsilon\right)^{s/d}  + \beta^{(m)}_j(V_m-\epsilon) 
            \right).
        \end{aligned}
    \end{equation}

    From Lemma~\ref{lem:wt_short_range}, we further have an upper estimate on the asymptotics of $ \U(\omega_N^*; B_j) $. 
    Recall that, by \eqref{eq:upper_semi} and the choice of $ r_x $, for each $ B_j^{(m)} $ there exists a closed subset $ S_j^{(m)} \subset B_j^{(m)} $ satisfying  $ \h_d(S_j^{(m)}) \geq (1-\epsilon) \h_d [B_j^{(m)}] $,  for which $ w(y,z) \leq w_m + \epsilon $ and $ V(y) \leq V_m + \epsilon $ whenever $ y,z \in S_j^{(m)} $.
    If $ w_m $, $ m=1,2 $, are both finite, Lemma~\ref{lem:wt_short_range} applies to sets $ S_j^{(m)} $. It gives for any pair of positive numbers $ \alpha_j^{(m)} $, $ m=1,2 $, for which $ \sum_m \alpha_j^{(m)} = \sum_m \beta_j^{(m)} $:
    \begin{equation}
        \label{eq:upper}
        \begin{aligned}
            \limsup_{N\to \infty}  \frac{\U(\omega_N^*; B_j)}{N^{1+s/d}}    
            \leq 
            &\sum_{m=1,2} 
            \left(
                ( w_m+\epsilon )\cdot \frac{C_{s,d}^k({\alpha^{(m)}_j})^{1+s/d}}{\h_d(S_j^{(m)})^{s/d}} + \alpha^{(m)}_j(V_m+\epsilon) 
            \right)\\
            \leq 
            &\sum_{m=1,2} 
            \left(
                \frac{ w_m+\epsilon }{(1-\epsilon)^{s/d}} \cdot \frac{C_{s,d}^k({\alpha^{(m)}_j})^{1+s/d}}{\h_d(B_j^{(m)})^{s/d}} + \alpha^{(m)}_j(V_m+\epsilon) 
            \right).
        \end{aligned}
    \end{equation}
    Note also that for the above to hold, we do not need $ w_m + V_m $ to be finite. For instance, suppose $ w_1 + V_1 <+\infty = w_2 + V_2 $; then the above inequality is trivial unless $ \alpha_j^{(2)} = 0 $, in which case we apply the argument of Lemma~\ref{lem:wt_short_range} to the ball $ B_j^{(1)} $ only.

    Inequality \eqref{eq:upper} holds in particular if the values  $ \alpha_j^{(m)} = \tilde\alpha_j^{(m)} $ are chosen to minimize the right-hand side over positive $ \alpha_j^{(m)} $, $ m=1,2 $, with $ \sum_m \alpha_j^{(m)} = \sum_m \beta_j^{(m)} $. Using Lagrange multipliers, we see that such $ \tilde \alpha_j^{(m)} $ must satisfy
    \[
        \frac{V_2- V_1}{C_{s,d}^k (1+s/d)} =
        (w_1+\epsilon)\left(\frac{\tilde \alpha_j^{(1)}}{(1-\epsilon) \h_d(B_j^{(1)})}\right)^{s/d} 
        -
        (w_2+\epsilon)\left(\frac{\tilde \alpha_j^{(2)}}{(1-\epsilon) \h_d(B_j^{(2)})}\right)^{s/d}.
    \]
    Note that the left-hand side in the above equation is independent of $ \h_d(B_j^{(1)}) / \h_d(B_j^{(2)}) $. As a result, 
    limit of the right-hand side for $ j\to \infty $ is also independent of this ratio.
    This fact will be essential in completing the proof.

    Observe that equations~\eqref{eq:lower}--\eqref{eq:upper} hold for every pair of sufficiently small radii $ r_j^{(m)} $.
    To obtain estimates for the density $ d\mu/d\h_d $, divide~\eqref{eq:lower} and \eqref{eq:upper} through by $ \h_d(B_j) $ and take $ j\to \infty $. Without loss of generality, the limits $ \gamma_m := \lim_{j\to \infty} \h_d(B_j^{(m)}) / \h_d(B_j) $ exist; otherwise we pass to a suitable subsequence. We have from \eqref{eq:lower}--\eqref{eq:upper} and optimality of $\tilde \alpha_j^{(m)} $,
    \begin{equation}
        \label{eq:double_estimates}
        \begin{aligned}
            &\sum_{m=1,2} \gamma_m \left(C_{s,d}^k(w_m-\epsilon)(\rho_m-\epsilon)^{1+s/d} + (\rho_m-\epsilon)(V_m-\epsilon)\right) \\
            &\leq
            \sum_{m=1,2} 
            \gamma_m \left(
                C_{s,d}^k \frac{ w_m+\epsilon }{(1-\epsilon)^{s/d}}\,\alpha_m^{1+s/d} + \alpha_m (V_m+\epsilon) 
            \right)\\
            &\leq\sum_{m=1,2} \gamma_m \left(C_{s,d}^k\frac{w_m+\epsilon}{(1-\epsilon)^{s/d}}\,\rho_m^{1+s/d} + \rho_m(V_m+\epsilon)\right), \\
        \end{aligned}
    \end{equation}
    where we denote $ \alpha_m := \lim_{j\to \infty} {{\tilde\alpha}^{(m)}_j} / \h_d(B_j^{(m)}) $; we ensure these limits exist by passing to a subsequence.
    Here $ \rho_m = d\mu/d\h_d(x_m) $, $ m=1,2 $. Since the above holds for every fixed $ \epsilon > 0 $, after one takes $ \epsilon\downarrow0 $, the inequalities turn into equalities: 
    \begin{equation}
        \label{eq:equalities}
        \begin{aligned}
            \sum_{m=1,2} \gamma_m \left(C_{s,d}^k w_m\rho_m^{1+s/d} + \rho_m\, V_m\right) 
            =
            \sum_{m=1,2} 
            \gamma_m \left(
                C_{s,d}^k w_m\alpha_m^{1+s/d} +\alpha_m\, V_m 
            \right).
        \end{aligned}
    \end{equation}
    Recall that the limit of ratios $ {{\tilde\alpha}^{(m)}_j} / \h_d(B_j^{(m)}) $ is independent of the limit of the ratio $ \h_d(B_j^{(1)}) /\h_d(B_j^{(2)}) $. On the other hand, the double estimate \eqref{eq:double_estimates} holds for any pair of sufficiently small balls $ B_j^{(m)} $. This allows to vary their radii independently, to produce sequences of balls, centered around $ x_1 $ and $ x_2 $, for which the limiting ratios $ (\gamma_1,\gamma_2) $ are $ (1,0) $ and (0,1). For such sequences, equation~\eqref{eq:equalities} gives 
    \[
        \rho_m = \alpha_m = \lim_{j\to \infty} \frac{\tilde \alpha_j^{(m)}}{\h_d(B_j^{(m)})}, \qquad m=1,2,
    \]
    whence we conclude that  the equation
    \[
        \frac{V_2 - V_1}{C_{s,d}^k(1+s/d)} = w_1\rho_1^{s/d} - w_2\rho_2^{s/d}
    \]
    holds for $ \h_d $-a.e.\ pair $ x_1, x_2 $. In particular, $ w(x,x)\rho(x)^{s/d} + V_1/(C_{s,d}^k(1+s/d)) ={const} =: L_1 < \infty  $ $ \h_d $-a.e., since we can pick $ x_1 $ among the points for which $ w(x_1, x_1) + V(x_1) < \infty $ and $ \rho(x_1) < \infty $. Combined with the condition $ \int \rho(x) \,d\h_d(x) = 1  $ that the function $ \rho(x)=d\mu/d\h_d $ must satisfy as the density of a probability measure, this yields~\eqref{eq:density}.

    In the remaining part of the proof we derive the formula for the asymptotics of minimizers of $ E_s^k $ on $ A $. To begin, note that when $ w(x,x) + V(x) = +\infty $ for $ \h_d $-a.e. $ x\in A $, using Proposition~\ref{prop:abs_cont} and arguing as in~\eqref{eq:lower}, we immediately have that the asymptotics with respect to $ N^{1+s/d} $ are infinite. It suffices to assume for the rest of this proof that $ w(x,x) + V(x) < \infty $ on a set of positive $ \h_d $-measure.

    By the above argument, $ w(x,x)\rho(x)^{s/d} + V(x) $ is bounded on $ \supp \mu $ by $ L_1 $; hence, $ w\rho^{1+s/d}\in L^1(A,\h_d) $; similarly, $ V\rho\in L^1(A,\h_d) $. As a result, $ \h_d $-a.e.\ point in $ A $ is a Lebesgue point for functions $ w\rho^{1+s/d} $ and $ V\rho $, and the measure $ \h_d $: 
    for any fixed $ \epsilon >0 $, at $ \h_d $-a.e.\ $ x\in A $ there exists a small enough $ r>0 $ such that 
    \begin{equation}
        \label{eq:lebesgue_point}
        \begin{aligned}
        \left|w(x,x)\rho(x)^{1+s/d} \cdot \h_d[B_A(x,r)] - \int_{B_A(x,r)}  w(y,y)\rho(y)^{1+s/d}\, d\h_d(y) \right| 
        &\leq \epsilon\h_d[B_A(x,r)],\\
        \left|V(x)\rho(x) \cdot \h_d[B_A(x,r)] - \int_{B_A(x,r)}  V(y)\rho(y)\, d\h_d(y) \right| 
        &\leq \epsilon\h_d[B_A(x,r)].
    \end{aligned}
    \end{equation}
    To obtain the expression for optimal asymptotics, we use \eqref{eq:lebesgue_point}, the second display in \eqref{eq:centers}, and argue as in~\eqref{eq:lower}, to derive for every $ \epsilon > 0 $ and $ \h_d $-a.e. $ x\in A $, with $ r< r_{x,\epsilon} $ sufficiently small:
    \[
        \begin{aligned}
            \liminf_{N\to \infty}&  \frac{\U(\omega_N^*; B_x)}{N^{1+s/d}}         
             \geq 
            \h_d(B_x) \left(C_{s,d}^k(w(x,x)-\epsilon )\cdot  \left( (1- \epsilon)\rho(x) \right)^{1+s/d}  + (1- \epsilon)\rho(x)(V(x)-\epsilon)\right) \\
            & \geq 
            (1-\epsilon)^{1+s/d}\int_{B_x} \left(C_{s,d}^k w(y,y) \rho(y) ^{s/d} + V(y)\right)\, d\mu(y) - \epsilon\h_d(B_x)\left(C_{s,d}^k \rho(x)^{1+s/d} + \rho(x) + 2 \right),
        \end{aligned}
    \]
    where $ B_x:= B_A(x,r) $. Using the Vitali covering theorem \cite[Theorem 2.8]{mattila1995geometry} for the Radon measure $ \h_d $, we can cover $ \h_d $-a.e.\ of $ A $ with a countable collection of such disjoint $ B_x $; since 
    \[
        \begin{aligned}
        \liminf_{N\to \infty} \frac{E_s^k(\omega_N^*; w,V)}{N^{1+s/d}} 
        &\geq \liminf_{N\to \infty} \sum_j  \frac{\U(\omega_N^*; B_{x_j})}{N^{1+s/d}} \\
        &\geq (1-\epsilon)^{1+s/d} \int_A \left(C_{s,d}^kw(x,x)\rho(x)^{s/d} + V(x)\right)\, d\mu(x)  -c\epsilon
        \end{aligned}
    \]
    for a suitable constant $ c $ (we used that $ \rho^{s/d} $ is bounded because $w\geq w_0$), it remains to show that $ \int (C_{s,d}^k w\rho^{s/d} + V)\,d\mu  $ is also an upper bound for the asymptotics. 

    Such upper bound follows by placing optimal configurations of cardinalities $ \lceil \mu(B_{x_j}) N \rceil $ into the sets $ S_{x_j} \subset B_A(x_j, \gamma r_{x_j,\epsilon}) $, defined in the same way as $ S_j^{(m)} $ above, for $ \gamma \in (0,1) $. Indeed, for any finite collection of disjoint closed balls $ B_m $ with $ \h_d(\partial B_m \cap A) = 0 $, by placing the minimizers in a suitable closed subset $ S_m \subset B_m $ satisfying \eqref{eq:lower_semi}, \eqref{eq:upper_semi}, with $ \h_d(S_m) \geq (1-\epsilon) \h_d(B_m) $:
    \[
        \omega_N := \bigcup_{m=1}^M \omega_{N_m}, \qquad E_s^k(\omega_{N_m};w,V)\leq \mathcal E_s^k({N_m}, S_m;w,V)+1, \quad N_m = \lceil \mu(B_m) N \rceil,
    \] 
    and arguing as in Lemma~\ref{lem:wt_short_range} one has
    \begin{equation*}
        \begin{aligned}
            \limsup_{N\to \infty}  \frac{E_s^k(\omega_N^*;w,V)}{N^{1+s/d}}
            &\leq \limsup_{N\to \infty}  \frac{E_s^k(\omega_N;w,V)}{N^{1+s/d}}\\
            &\leq\sum_{m=1}^M  \left(C_{s,d}^k\frac{w_m+\epsilon}{(1-\epsilon)^{s/d}}((1+\epsilon)\rho_m)^{1+s/d} + (1+\epsilon)\rho_m(V_m+\epsilon)\right) \h_d(B_m),
        \end{aligned}
    \end{equation*}
    where as usual, we write $ w_m $ and $ V_m $ for the values of the respective functions at the centers of $ B_m $, and $ B = \bigcup_m B_m $. Choosing the centers of $ B_m $ in $ \supp \mu $
    and using~\eqref{eq:lebesgue_point} gives
    \[
        \limsup_{N\to \infty}  \frac{E_s^k(\omega_N^*;w,V)}{N^{1+s/d}} \leq \frac{(1+\epsilon)^{1+s/d}}{(1-\epsilon)^{s/d}} \int_B \left(C_{s,d}^k w(x,x)\rho(x)^{s/d}+ V(x)\right)\, d\mu(x)  + \epsilon c ,
    \]
    where it is used again that $ \rho^{s/d} $ is bounded on $ \supp \mu $. This finishes the proof of the theorem.
\end{proof}

\section{Connections to other short-range interactions}
\label{sec:connections}
\subsection{Convex kernels on the circle}
When $ d = \dim_H A = 1 $, we can compute explicitly the values of $ C_{s,1}^k $ for any $ s > 0 $ and $ k\geq 1 $. Moreover, we will show that on the periodized interval $ [0,1] $, the minimizers of the energy $ E^k_\phi $ defined below are equally spaced, for any convex decreasing function of distance $ \phi $. Equivalently, minimizers of such energies on $ \mathbb S^1 $ with embedded distance are equally spaced points for any convex decreasing kernel.

\begin{theorem}
    Let $  A = \mathbb S^1  $  with the distance $ \vartheta = s / 2\pi $ for the arc length $ s $, and assume that $ \phi:[0,1/2]\to [0,\infty] $ is a convex decreasing function. For any $ N\geq 1$ and $ k\geq 1 $, the energy
    \[
        E^k_\phi(\omega_N) := \sum_{x\in\omega_N} \sum_{y\in\nn_k(x;\omega_N)} \phi\left(\vartheta(x , y) \right) 
    \]
    is minimized by every set $ \omega_N^* $ consisting of $ N $ equally spaced points.
\end{theorem}
\begin{proof}
    Consider an arbitrary set $ \omega_N $ of $ N $ distinct points in $  \mathbb S^1 $. It suffices to show that its energy is at least the one of $ \omega_N^* $, as defined above.
    We will assume that the entries of $ \omega_N = (x_1,\ldots,x_N) $ are numbered clockwise, so that for example $ x_1 $ and $ x_3 $ are adjacent to the point $ x_2 $, etc. We will also use indices of $ x_i $ modulo $ N $, so for any $ x_i $ the two adjacent points in the above ordering are given by $ x_{i-1} $ and $ x_{i+1} $.

    Consider the following sets of $ k $ indices 
    \[
        I_{i,k} := \left\{ i-\left\lfloor \frac k2\right\rfloor, i-\left\lfloor \frac k2\right\rfloor +1,\ldots, i-1, i+1,\ldots, i+\left\lceil \frac k2\right\rceil-1, i+\left\lceil \frac k2\right\rceil \right\}.  
    \] 
    Order the points in $  \{ x_j \in \omega_N : j \in I_{i,k} \} $ by the nondecreasing distance from $ x_i $ and denote the points with the resulting ordering
    $ y^{(i)}_1, \ldots, y^{(i)}_k $. Then there holds
    \[
        \vartheta(x_i , (x_i;\omega_N)_j) \leq  \vartheta(x_i , y^{(i)}_j),
    \]
    where as before, $ (x_i;\omega_N)_j  $ is the $ j $-th nearest entry of $ \omega_N $ to $ x_i $. This inequality follows from $ \vartheta(x_i , y^{(i)}_l) \leq  \vartheta(x_i , y^{(i)}_j) $ for $ l\leq j $, so there are at least $ j-1 $ entries of $ \omega_N $ that are closer to $ x_i $ than $ y^{(i)}_j $.
    By the monotonicity of $ \phi $, then
    \begin{equation}
        \label{eq:iprime}
        \sum_{y\in \nn_k(x_i;\omega_N)} \phi\left(\vartheta({x}_i , y) \right) \geq \sum_{j=1}^k \phi\left( \vartheta ({x}_i , y^{(i)}_j) \right), \qquad 1\leq i \leq N.  
    \end{equation}

    Now observe that for any configuration of $ N $ distinct points $ \omega_N \in (\mathbb S^1)^N $ numbered clockwise,
    \[
        \sum_{i=1}^N \vartheta(x_i , x_{i+1}) = 1.
    \]
    Indeed, the above sum contains the distances between adjacent points, which add up to the length of $  \mathbb S^1 $. Further, one has
    \begin{equation}
        \label{eq:distance_sum}
        \sum_{i=1}^N \vartheta(x_i , x_{i+j})  = \sum_{i=1}^N \sum_{l=1}^j  \vartheta( x_{i+l-1} , x_{i+l}) = j,
    \end{equation}
    whenever $ 2j \leq N $. A similar formula holds for negative $ j $.

    In view of \eqref{eq:iprime}, \eqref{eq:distance_sum}, convexity of $ \phi $, and that without loss of generality $ k\leq N-1 $, we obtain
    \[
        \begin{aligned}
        E^k_\phi(\omega_N)
        &= \sum_{i=1}^N \sum_{y\in \nn_k(x_i;\omega_N)} \phi\left( \vartheta({x}_i , y) \right) \geq \sum_{i=1}^N \sum_{j=1}^k \phi\left( \vartheta ({x}_i , y^{(i)}_j) \right) \\
        &= \sum_{\substack{j=-\lfloor k/2\rfloor\\ j\neq 0}}^{\lceil k/2 \rceil} \sum_{i=1}^N \phi\left(\vartheta ({x}_i , x_{i+j} )\right) \geq \sum_{\substack{j=-\lfloor k/2\rfloor\\ j\neq 0}}^{\lceil k/2 \rceil} N \, \phi \left(\frac1N \sum_{i=1}^N \left(\vartheta({x}_i , x_{i+j}) \right)\right) \\
        &= N  \sum_{\substack{j=-\lfloor k/2\rfloor\\ j\neq 0}}^{\lceil k/2 \rceil} \phi \left(\frac {|j|}N \right) = E^k_\phi(\omega^*_N).
    \end{aligned}
    \]
    In the second line of the above display we used Jensen's inequality. Here, as defined in the statement of the theorem, $ \omega_N^* $ consists of $ N $ equally spaced points in $  \mathbb S^1 $.
\end{proof}

\begin{corollary} 
    The value of the constant $ C_{s,1}^k $ is given by 
    \[
        C_{s,1}^k  = \sum_{\substack{j=-\lfloor k/2\rfloor\\ j\neq 0}}^{\lceil k/2 \rceil} \frac1{|j|^s}.
    \] 
\end{corollary}
\begin{proof}
    The unit circle $ \mathbb S^1 $ with the metric $ \vartheta $ above can be identified with the periodized unit interval $ [0,1) $ equipped with the natural distance. Due to the short-range properties of Riesz $ k $-energies $ E^k_s $ (with convex decreasing $ \phi(r) = 1/r^s $), the asymptotics of the minimal energy for set $ A' = [0,1) $ with this distance coincide with the asymptotics for $ A = [0,1] \subset \mathbb R $ with the Euclidean distance. 
\end{proof}

\subsection{Relation of \texorpdfstring{$ k $}{k}-energies with \texorpdfstring{$ s>d $}{s>d} to full hypersingular Riesz energies} 

As explained in Section~\ref{subsec:relationHyper}, we obtain our results about the full interaction in the hypersingular case $ s>d $ by passing to the limit $ k\to \infty $ in $ E_s^k $. To do that, we will need the following lemma, which has been established in a slightly different form in \cite[Lem.\ 5.2]{borodachovLow2014}. Recall that $ \Delta(\omega_N) = \min_{1\leq i < j \leq N} \|x_i - x_j\| $ stands for the separation of the configuration $ \omega_N $.
\begin{lemma}
    \label{lem:offdiagonal}
    Let $ A \subset \mathbb R^d $ be a compact set.
    Let further $ \omega_N \subset A $ be a sequence of configurations satisfying $ \Delta(\omega_N) \geq c_0 N^{-1/d} $, $ N\geq 2 $, and $ w $ a bounded weight function on $ A\times A $. Then there holds
    \[
        \limsup_{N\to \infty}\frac1{N^{1+s/d}}\sum_{x\in \omega_N} \sum_{y\notin \nn_k(x;\omega_N) } w(x,y)\|x - y \|^{-s} \leq c(k,s,d), 
    \]
    with $ c(k,s,d) \to 0 $, $ k\to \infty $.
\end{lemma}
\begin{proof}
    Fix a point $ x\in\omega_N $. 
    Denote $ 2h_N := c_0 N^{-1/d}$ for brevity. For an $ m\geq1 $, let 
    \[
        L_m =  \{ y \in \omega_N : mh_N < \|y - x\| \leq (m+1)h_N \}.
    \]
    There holds $ \h_d[B(x_i,r) ] =  v_d r^d $ for any $ r>0 $, whence $ \h_d [B(x_i,r+t) \setminus B(x_i,r)]  \leq c_1 t(r+t)^{d-1} $, $ t\geq 0 $ for some positive constant $ c_1=c_1(d) $. Since the distance between any two points in $ \omega_N $ is at least $ 2h_N $, the interiors of balls $ B(x_j, h_N) $ for $ 1\leq j \leq N $ must be pairwise disjoint. This allows to estimate $ \#L_m $ by volume considerations: since
    \[
        \bigcup_{y\in L_m} B(y, h_N) \subset \left[B\left(x,(m+2)h_N\right) \, \big\backslash \, B\left(x,(m-1)h_N\right)\right],
    \]
    there holds $ v_d h_N^d \cdot \#L_m \leq  3c_1 h_N((m+2)h_N)^{d-1}$, which gives for $c_2=c_2(d)$,
    \[ 
        \#L_m \leq  c_2 m^{d-1}.
    \]
    Summing up the pairwise energies over spherical layers around $ x $, one obtains further
    \[ 
        \begin{aligned}
            \sum_{y\in\omega_N} \|{y} - x \|^{-s} 
            &=  \sum_{m=1}^\infty \sum_{y\in L_m} \|y - {x} \|^{-s} \leq \sum_{m=1}^\infty \frac{c_2 m^{d-1}}{(m h_N)^s} = \frac{c_2 2^s N^{s/d}}{c_0^s} \sum_{m=1}^\infty \frac1{m^{s-d+1}}.
        \end{aligned}
    \]
    This implies for $k \geq \sum_{m=1}^{M-1} c_2 m^{d-1} \geq \sum_{1}^{M-1} \# L_m$ that
    \[
        \frac1{N^{1+s/d}}\sum_{x\in\omega_N}  \sum_{y\notin \nn_k(x;\omega_N)} \|x-y\|^{-s} \leq c_2(2/c_0)^{s} \sum_{m=M}^\infty \frac1{m^{s-d+1}},
    \]
    which converges to $ 0 $ for $ k\to \infty $, and thus gives the desired statement. Observe that the convergence is uniform over all configurations with $ \Delta(\omega_N) \geq c_0 N^{-1/d} $. 
\end{proof}

\begin{lemma}
    \label{lem:k_to_infty}
    Suppose $ A\subset \mathbb R^d $ is a compact set, $ w, V $ satisfy the assumptions of Theorem~\ref{thm:separation}, and $w$ is bounded; assume also a sequence $ k_n $, $ n\geq 1 $, satisfies $ k_n\to \infty $, $ n\to \infty $. Then 
    \[
        \mathcal E^{k_n}_s(A,N; w,V) \Big/\mathcal E_s(A,N; w, V)  \longrightarrow 1, \qquad N\to\infty, n\to \infty.
    \]
\end{lemma}
\begin{proof}
    Let $ \omega_N^* = \{ x_1^*,\ldots,x_N^* \} $ be such that
    \[
        E^{k_n}_s(\omega_N^*; w, V) \leq \mathcal E^{k_n}_s(\omega_N; w, V) + 1, \qquad N \geq 2.
    \]
    Similarly, let $ \omega_N' = \{ x_1',\ldots,x_N' \} $ be a sequence near-minimizing $ E_s $:
    \[
        E_s(\omega_N'; w, V) \leq \mathcal E_s(A, N; w,  V) + 1, \qquad N \geq 2. 
    \]
    By the construction of $ \omega_N^* $ and $ \omega_N' $, for every $ N $ there holds, 
    \[ 
        E^{k_n}_s(\omega_N^*; w, V) \leq E^{k_n}_s(\omega_N'; w, V) + 1 \leq E_s(\omega_N'; w, V) + 1 \leq E_s(\omega_N^*; w, V) + 2.
    \]
    In addition, since $ \omega_N^* $ is separated by Theorem~\ref{thm:separation}, using Lemma~\ref{lem:offdiagonal} the difference $ \mathcal E_s(A,N; w, V) - \mathcal E^{k_n}_s(A,N; w,V)$ can be estimated by
    \[
        \begin{aligned}
            E_s(\omega_N^*; w, V)  &-  E^{k_n}_s(\omega_N^*; w, V)\\
            &\leq
            \sum_{x\in \omega_N} \sum_{y\notin \nn_{k_n}(x;\omega_N) } w(x,y)\|x - y \|^{-s} = c(k_n,s,d)N^{1+s/d},
    \end{aligned}
    \]
    where $ c(k_n,s,d) \to 0 $, $ k_n\to \infty $, and we used the boundedness of $ w $. This completes the proof of the lemma.
\end{proof}
\begin{proof}[Proof of Theorem~\ref{thm:s_ge_d_asympt}.] 
    If $ w(x,x) +V(x) $ is not bounded on a subset of $ A $ of positive $ \h_d $-measure, the optimal asymptotics of $ E_s^1 $ are infinite by Theorem~\ref{thm:k_asympt}, and since $ E_s\geq E_s^1 $, there is nothing to prove.

    For a compact $ A\subset \mathbb R^d $, constant weight $ w $, and a lower semicontinuous $ V$, the first claim of Theorem~\eqref{thm:s_ge_d_asympt} follows from Lemma \ref{lem:k_to_infty}. The asymptotics and limiting density of asymptotic minimizers of $ E_s $ are obtained by passing to the limit in Theorem~\ref{thm:k_asympt} and the dominated convergence theorem. To extend the result to a compact $ (\h_d,d) $-rectifiable $ A \subset \mathbb R^p $, we then apply Lemma~\ref{lem:federer} to the functionals $ E_s^k $  and $ E_s $. Note that the short-range property and stability for $ E_s $ for $ s>d $ are well-known \cite[Section 8.6.2]{borodachovDiscrete2019}. The case of general weight and external field follows by extending the asymptotics of $ E_s $ according to the argument given in the proof of Theorem~\ref{thm:k_asympt} and monotonic pointwise convergence due to the factor $C_{s,d}^k \uparrow C_{s,d} $, \(k\to \infty \), in the resulting integral functionals expressing the asymptotics. The limiting distribution is likewise obtained by applying the argument in the proof of Theorem~\ref{thm:k_asympt}. 
\end{proof}

\subsection{Proof of \texorpdfstring{$ \Gamma $}{Gamma}-convergence}

For a compact $ A $ we denoted by  $ \mathcal P(A) $ the space of probability measures supported on $ A $. It is a  compact metrizable space. As explained in the introduction, we discuss the properties of short-range interactions on discrete configurations $ \omega_N $, by viewing them as acting on the normalized counting measures $ \nu(\omega_N) \in \mathcal P(A) $.

 The sequence introduced in \ref{it:recovery_seq} is called a \textit{recovery sequence} at the point $ x $. Usefulness of $ \Gamma $-convergence for energy minimization consists in that, together with compactness of $ X = \mathcal P(A) $, it guarantees that minimizers of $ F_N $ converge to those of $ F $. Moreover, $ F_N $ need not attain its minimizer, but this is the case for $ F $ on compact sets, due to its lower semicontinuity.  Namely, the following properties hold.
 \begin{proposition}[\cite{braidesLocal2014}, \cite{dalmasoIntroduction1993}]
    If a sequence of functionals $ \{F_N\} $ on a compact metric space $ X $  $ \Gamma $-converges to $ F $, then
    \begin{enumerate}
        \item $ F $ is lower semicontinuous and $ \min F = \lim_{N\to \infty}\inf F_N $
        \item if $ \{x_N\} $  is a sequence of (global) minimizers of $ F_N $, converging to an $ x\in X $, then $ x $ is a (global) minimizer for $ F $.
    \end{enumerate}
    If $ F_N $ is a constant sequence, $ \glim F $ is the lower semicontinuous envelope of $ F $; i.e., the supremum of lower semicontinuous functions bounded by $ F $ above.
\end{proposition}

\begin{proof}[Proof of Theorem~\ref{thm:gamma}.]
    To verify the property \ref{it:gamma_lower} of the definition of $ \Gamma $-convergence, suppose a sequence $ \{\mu_N\}\subset \mathcal P({A}) $ weak$ ^* $ converges to $ \mu\in\mathcal P({A}) $. Observe that if     
    \begin{equation*}
        \liminf_{N\to\infty} \frac1{N^{1+s/d}} \f_N(\mu_N; w, V) = +\infty \geq \f(\mu; w, V),
    \end{equation*}
    the inequality in \ref{it:gamma_lower} holds trivially. It therefore suffices to assume that the limit in the last equation is finite. In particular, $ \{\mu_N\} $ must contain a subsequence comprising only elements from $ \mathcal P_N({A}) $, so without loss of generality we suppose that $ \mu_N$,  $ N\geq 1, $ is a sequence of discrete measures converging to $ \mu \in \mathcal P(A) $, such that the following limit exists and is finite:
    \[
        \lim_{N\to\infty} \frac1{N^{1+s/d}} \f_N(\mu_N; w, V),
    \]
    so that it will suffice to show its value is at least $ \f(\mu; w, V) $.
    By Proposition~\ref{prop:abs_cont}, finiteness of the asymptotics implies that $ \mu $ must be absolutely continuous with respect to $ \h_d $. 

    The rest of the proof can be obtained by a modification of that of Theorem~\ref{thm:k_asympt}. Indeed, let $ \omega_N $ be the sequence of $ N $-point configurations corresponding to the measures $ \mu_N $, and denote $ \rho := d\mu /d\h_d $. First, let $ \rho, w $, $ V $ be bounded on $ A $. Then $ w\rho^{1+s/d} $, $ V $ are also bounded  and  hence in $ L^1(A,\h_d) $, and thus equations \eqref{eq:lebesgue_point} apply. Since the argument resulting in~\eqref{eq:lower} did not use optimality of the sequence of configurations, it applies to the $ \omega_N $; thus, we have for $ \h_d $-a.e. $ x\in A $, setting $ B_x:= B_A(x,r) $ with $ r< r_{x,\epsilon} $ sufficiently small:
    \[
        \begin{aligned}
            \liminf_{N\to \infty}&  \frac{\U(\omega_N; B_x)}{N^{1+s/d}}         
             \geq 
            \h_d(B_x) \left(C_{s,d}^k(w(x,x)-\epsilon )\cdot  \left( (1- \epsilon)\rho(x) \right)^{1+s/d}  + (1- \epsilon)\rho(x)(V(x)-\epsilon)\right) \\
            & \geq 
            (1-\epsilon)^{1+s/d}\int_{B_x} \left(C_{s,d}^k w(y,y) \rho(y) ^{s/d} + V(y)\right)\, d\mu(y) - \epsilon\h_d(B_x)\left(C_{s,d}^k \rho(x)^{1+s/d} + \rho(x) + 2 \right),
        \end{aligned}
    \]
    Applying Vitali covering theorem to $ A $, we conclude as in the proof of Theorem~\ref{thm:k_asympt}:
    \begin{equation}
        \label{eq:lower_bound_pf}
        \begin{aligned}
        \liminf_{N\to \infty} \frac{E_s^k(\omega_N; w,V)}{N^{1+s/d}} 
        &\geq \liminf_{N\to \infty} \sum_j  \frac{\U(\omega_N; B_{x_j})}{N^{1+s/d}} \\
        &\geq (1-\epsilon)^{1+s/d} \int_A \left(C_{s,d}^kw(x,x)\rho(x)^{s/d} + V(x)\right)\, d\mu(x)  -c\epsilon
        \end{aligned}
    \end{equation}
    This completes the proof of \ref{it:gamma_lower} for bounded densities $ \rho $ and $ w,V $. The case of unbounded $\rho$ follows by superadditivity of $ E_s^k(\omega_N; w,V) $ as a function of $ \omega_N $, and the previous lower bound for the probability measures
    \[
        \mu_h(E) := \frac{\int_{E\cap \{ \rho \leq h \}} \rho(x)\, d\h_d(x) }{\int_{\{ \rho \leq h \}} \rho(x)\, d\h_d(x) }.
    \]
    In view of the property $ \rho \cdot \mathbbm{1}_{\rho\leq h} \uparrow \rho $ as $ h\to \infty $ for the (non-normalized) densities of $ \mu_h $, monotone convergence theorem applies to the integral in the right-hand side of~\eqref{eq:lower_bound_pf}, and the lower bound with $ \mu $ in the integral follows by approximation. The cases of unbounded weights $ w $ and external fields $ V $ are similarly handled by considering the finite truncations $ w_h:= w\cdot \mathbbm{1}_{w\leq h} $ and $ V_h:= V\cdot \mathbbm{1}_{V\leq h} $, and using the monotone convergence theorem in the right-hand side of \eqref{eq:lower_bound_pf}. This proves~\ref{it:gamma_lower}.

    To present a recovery sequence for \ref{it:recovery_seq}, we again invoke the argument from the proof of Theorem~\ref{thm:k_asympt}. In the case of a bounded $ \rho $, constructing a sequence of piecewise minimizers $ \omega_N $ approximating $\rho$ as in that proof gives
    \[
        \limsup_{N\to \infty}  \frac{E_s^k(\omega_N;w,V)}{N^{1+s/d}} \leq \frac{(1+\epsilon)^{1+s/d}}{(1-\epsilon)^{s/d}} \int_B \left(C_{s,d}^k w(x,x)\rho(x)^{s/d}+ V(x)\right)\, d\mu(x)  + \epsilon c ,
    \]
    where $ B $ is a union of disjoint closed balls with $ \mu(B) > (1-\epsilon) $. In the case of unbounded $ \rho $, we construct recovery sequences for $ \mu_h $ as above, and then take a diagonal sequence.
\end{proof}

\bibliographystyle{acm}
\bibliography{refs}

\vskip\baselineskip

{\footnotesize
    {\sc Center for Constructive Approximation\\ \indent Department of Mathematics, Vanderbilt University, Nashville, TN, 37240}

    {\it Email address:} {\tt doug.hardin@vanderbilt.edu}

    {\it Email address:} {\tt edward.b.saff@vanderbilt.edu}

    \medskip
    {\sc Department of Mathematics, Florida State University, Tallahassee, FL 32306}\\
    \indent{\sc Current address: Department of Mathematics, Vanderbilt University, Nashville, TN, 37240}

    {\it Email address:} {\tt oleksandr.vlasiuk@vanderbilt.edu}

} 
\end{document}